\documentclass[10pt]{amsart}
\usepackage[utf8]{inputenc}
\usepackage[T1]{fontenc}
\usepackage[english]{babel}
\title{The Martin boundary of a free product of abelian groups}
\author{Matthieu Dussaule}
\date{}

\usepackage{silence}
\usepackage{comment}
\usepackage{amsmath}
\usepackage{amssymb}
\usepackage{dsfont}
\usepackage{amsfonts}
\usepackage{amsthm}
\usepackage{tikz}
\usepackage{graphicx}
\usepackage{fancyhdr}
\usepackage{caption}
\usepackage{enumitem}
\usepackage{mathtools}
\usepackage[colorlinks=true,pdfstartview=FitH,bookmarks=false]{hyperref}
\hypersetup{urlcolor=black,linkcolor=black,citecolor=black,colorlinks=true}

\newcommand\N{\mathbb{N}}
\newcommand\Z{\mathbb{Z}}

\newcommand\R{\mathbb{R}}

\newcommand\Ss{\mathds{S}}

\theoremstyle{definition}
\newtheorem{hyp}{Assumption}
\newtheorem*{hyp*}{Assumption 2'}

\theoremstyle{plain}
\newtheorem{definition}{Definition}[section]
\newtheorem{prop}[definition]{Proposition}
\newtheorem{coro}[definition]{Corollary}
\newtheorem{theorem}[definition]{Theorem}
\newtheorem{lem}[definition]{Lemma}
\newtheorem*{thm1'*}{Theorem 1.1'}
\newtheorem*{thm1''*}{Theorem 1.1''}
\newtheorem*{prop*}{Proposition}
\newtheorem*{lem*}{Lemma}

\theoremstyle{remark}
\newtheorem{rem}{Remark}[section]

\newtheorem*{rem*}{Remark}

\begin{document}

\begin{abstract}
Given  a  probability  measure  on  a  finitely  generated  group,  its  Martin
boundary is a way to compactify the group using the Green function
of the corresponding random walk.
It is known from the work of W.~Woess that when a finitely supported random walk on a free product of abelian groups is adapted to the free product structure,
the Martin boundary coincides with the geometric boundary.
The main goal of this paper is to deal with non-adapted finitely supported random walks,
for which there is no explicit formula for the Green function.
Nevertheless, we show that the Martin boundary still coincides with the geometric boundary.
We also prove that the Martin boundary is minimal.
\end{abstract}

\maketitle

\section{Introduction}\label{SectionIntroduction}
The main goal of this article is to get a full description of the Martin boundary for a finitely supported random walk in a free product $\Z^{d_1}\star \Z^{d_2}$, or more generally a free product of finitely many virtually abelian groups,
without assuming the random walk is adapted to the free product structure.
In all the paper, we assume for simplicity that the ranks of the abelian groups are positive, i.e.\ $d_1,d_2\geq 1$, although our techniques also work when $d_1,d_2\geq 0$, except for the particular case of the group $\Z/2\Z\star \Z/2\Z$.

\subsection{Random walks on free products}\label{SectionRandomwalksonfreeproducts}
Consider a Markov chain on a countable space $E$, with transition kernel $p$,
and assume that this Markov chain is transitive and transient.
The Martin boundary of $E$ with respect to $p$ is a way to compactify $E$ turning the probabilistic behaviour encoded in $p$ into geometric information.
It will be properly defined in Section~\ref{SectionMarkovchainsandMartinboundary}, along with the minimal Martin boundary which is a subspace of it.
The Martin boundary provides an actual compactification and contains topological information, and thus differs from the Poisson boundary which is only a measurable boundary.
In the spaces we study in this article, there already exists a geometric boundary (see Section~\ref{SectionGeometricboundaries} for proper definitions).
It is thus reasonable to ask whether the Martin boundary coincides with the geometric one.

We will essentially deal with spaces $E$ together with a finitely generated group $\Gamma$ and a group action $\Gamma \curvearrowright E$.
We will then assume that $p$ is homogeneous, that is, invariant under the group action.
In the special case where $E=\Gamma$ and the action is by translation, the Markov chain is called a random walk.
We will also make a technical assumption, related to sufficiently large exponential moments.
In this introduction, the reader should think of a finitely supported Markov chain (see Assumption~\ref{hyp1} and Lemma~\ref{lemmafinitesupport} below).

In the abelian case, the Martin boundary is a sphere at infinity when the random walk is non-centered and it is trivial otherwise.
In both cases, the Martin boundary is minimal.
This is essentially due to P.~Ney and F.~Spitzer (see \cite{NeySpitzer}).

In the nilpotent case, using the work of G.~Margulis on positive harmonic functions on nilpotent groups (see \cite{Margulis}), one can show that the minimal Martin boundary is the same as the minimal Martin boundary of the abelianized group.
Indeed, in this article, the author proves that minimal harmonic functions are constant on the cosets of the commutator subgroup, when the nilpotent group satisfies some technical assumption, which is automatically satisfied if the group is finitely generated (see Section~\ref{SectionMarkovchainsandMartinboundary} for the definition of minimal harmonic functions).
Thus, the study of the minimal Martin boundary of a nilpotent group reduces to the abelian case.
However, it is still not know to the author's knowledge if there is a geometric description of the full Martin boundary for any finitely generated nilpotent group, although
when the random walk is centered, the Martin boundary is trivial (see \cite{Alexopoulos}).
In \cite{HueberMuller}, H.~Hueber and D.~Müller show that the Martin boundary of a continuous random walk in the Heisenberg group $H^3(\R)$ is homeomorphic to a disc.
In this closely related setting, the full Martin boundary thus differs from the minimal one.

If the group is hyperbolic, then the Martin boundary coincides with the Gromov boundary of its Cayley graph and it is minimal.
This is due to A.~Ancona for finitely supported random walks (see \cite{Ancona}) and to S.~Gouëzel for walks with superexponential moments (see \cite{Gouezel2}).

In the following, we will deal with free products $\Z^{d_1}\star \Z^{d_2}$.
Those groups are hyperbolic relative to the subgroups $\Z^{d_1}$ and $\Z^{d_2}$.
One expects that the Martin boundary is, in some sense, the Gromov boundary in the hyperbolic part combined with spheres at infinity in abelian parts.
This is precisely the $\mathrm{CAT}(0)$ boundary of $\Z^{d_1}\star \Z^{d_2}$ we define in Section~\ref{SectionGeometricboundaries}.
W.~Woess proved that the Martin boundary is the $\mathrm{CAT}(0)$ boundary when the random walk is adapted to the free product structure (see \cite{Woess3}).
This means that the transition kernel can be written as
$$p(e,x)=\left\{
\begin{matrix}
        p_1(e,x)&  \text{ if }  x\in \Z^{d_1}\\
        p_2(e,x)&  \text{ if }  x\in \Z^{d_2}\\
        0&  \text{ otherwise}
\end{matrix}
    \right .$$
which basically means that at each step, one can either move on $\Z^{d_1}$ or move on $\Z^{d_2}$.
We will not need this assumption and the main result of this paper is the following.
The precise definitions of the Martin compactification and the $\mathrm{CAT}(0)$ compactification will be given in Section~\ref{SectionMarkovchainsandMartinboundary} and Section~\ref{SectionGeometricboundaries}.

\begin{theorem}\label{theorem1}
Consider an irreducible, finitely supported random walk on the free product $\Z^{d_1}\star \Z^{d_2}$, with $d_1,d_2\geq 1$.
Then, the Martin compactification coincides with the $\mathrm{CAT}(0)$ compactification and every point in the Martin boundary is minimal.
\end{theorem}

The technique of W.~Woess essentially reduces to the case of a nearest neighbor random walk.
In this context, the random walk has to go through particular points when it changes cosets.
Thus, the probability of going from $x$ to $y$ can be written as
$$\mathds{P}(x\rightarrow y)=\mathds{P}(x\rightarrow x_1)\cdots \mathds{P}(x_n\rightarrow y),$$
for suitable points $x_1,...,x_n$.
This technique does not apply when the random walk is not adapted to the free product structure.
Part of it can however be recovered using transitional sets (see Definition \ref{deftransitionalset} in Section~\ref{SectionRandomwalksinfreeproducts}) adapted from Y.~Derrienic (see \cite{Derriennic}).
It makes it possible to understand sequences going to infinity by changing cosets infinitely many times.

To deal with trajectories staying in the Euclidean parts $\Z^{d_i}$, we would like to restrict the random walk to those spaces and use results for random walks in $\Z^{d_i}$.
We introduce for that the transition kernel of the first return to the $\Z^{d_i}$ factors.
However, since we do not assume that the random walk is nearest neighbor,
the induced transition kernel is not finitely supported.
To avoid this issue, we do not restrict the random walk only to $\Z^{d_i}$ factors, but to neighborhoods of these.
Thus, we first have to describe the Martin boundary of thickenings of $\Z^{d}$ of the form $\Z^{d}\times \{1,...,N\}$.
This question has already been studied in the closely related setting of random walks in $\R^d\times X$, where $X$ is compact, by M.~Babillot in \cite{Babillot}.
Our techniques are similar.
However, we cannot apply directly her results.
Indeed, in our situation, the random walk on the free product is transient, so that restrictions to neighborhoods of $\Z^{d_i}$
do not have total probability 1.
Thus, we have to identify the Martin boundary of a non-probability transition kernel on $\Z^{d}\times \{1,...,N\}$.
This is done in Section~\ref{SectionSub-Markovchainsonathickenedlattice}.

The fact that the abelian groups are of the form $\Z^{d_i}$ is not really important and
our techniques allow us to state the following.

\begin{thm1'*}
Let $\Gamma_1$ and $\Gamma_2$ be two infinite finitely generated virtually abelian groups.
Consider an irreducible, finitely supported random walk on the free product $\Gamma_1\star \Gamma_2$.
Then, the Martin compactification coincides with the $\mathrm{CAT}(0)$ compactification and every point in the Martin boundary is minimal.
\end{thm1'*}

For simplicity, we will only prove Theorem~\ref{theorem1} and say a few words about Theorem~1.1'.
The main technical results to prove these theorems are Proposition~\ref{theorem1'spitzer}, Proposition~\ref{theorem2spitzer} and Proposition~\ref{latticeMartinboundary}.
They need some definitions to be properly stated and thus, are not part of this introduction.
Roughly speaking, the first two propositions give asymptotic estimates for the $n$th convolution power of the transition kernel and for the Green function, in the case of a $\Z^d$-invariant chain on $\Z^d\times \{1,...,N\}$.
The transition kernel is not necessarily a probability measure.
Proposition~\ref{latticeMartinboundary} gives a description of the Martin boundary in this situation, using Propositions~\ref{theorem1'spitzer} and~\ref{theorem2spitzer}.

\medskip
When studying random walks on $\Z^d$, one thing that matters is whether the random walk is centered of not.
Denoting by $p$ the transition kernel, the random walk is centered if
$$\sum_{x\in \Z^d}xp(0,x)=0.$$

For a $\Z^d$-invariant Markov chain on $\Z^{d}\times \{1,...,N\}$, given by a transition kernel $p$, the good definition of centering is as follows.
Let $\nu_0$ be the unique stationary measure for the induced chain on $\{1,...,N\}$.
We see $\nu_0$ as a $\R^N$ vector of coordinates $\nu_0(k)$, $1\leq k\leq N$.
We define the average horizontal displacement as
$$\overrightarrow{p}=\sum_{x\in \Z^d}\sum_{k,j\in \{1,...,N\}}\nu_0(k)xp((0,k),(x,j))=0.$$
We will say that the chain on $\Z^{d}\times \{1,...,N\}$ is centered if $\overrightarrow{p}=0$.
This is essentially an averaged version of the condition in $\Z^d$.
W.~Woess proved an analog of the law of large numbers,
stating that $\frac{1}{n}Y_n$ almost surely converges to $\overrightarrow{p}$, where $Y_n$ is the horizontal component of the Markov chain (see \cite[Theorem~ 6.7]{Woess1}).

In the particular case of a random walk in a group of the form $\Z^d\times L$, where $L$ is a finite group,  with neutral element $l_0$, this condition of centering takes the form
$$\sum_{(x,l)\in \Z^d\times L}xp((0,l_0),(x,l))=0.$$
Indeed, in this situation, the vector $\nu_0$ does not depend on $k$. This is because $\nu_0$ is a left eigenvector associated to the dominant eigenvalue of a stochastic matrix $F$, which is bi-stochastic in the case of a random walk on $\Z^d\times L$ (see Section~\ref{SectionSub-Markovchainsonathickenedlattice}).
Thus, our condition of centering coincides with the usual one (see for example \cite{Alexopoulos}) in this setting.

As a particular case of Proposition~\ref{latticeMartinboundary} we get the following.
\begin{theorem}\label{theorem2}
Consider an irreducible, $\Z^d$-invariant, finitely supported Markov chain on the thickened lattice $\Z^d\times \{1,...,N\}$ and assume that it is non-centered.
Then, the Martin compactification coincides with the $\mathrm{CAT}(0)$ compactification and every point in the Martin boundary is minimal.
\end{theorem}

We will also prove the following theorem, which is a generalization of a theorem of G.~P\'olya (see \cite{Polya}).

\begin{theorem}\label{theorem3}
Consider an irreducible, $\Z^d$-invariant, finitely supported Markov chain on the thickened lattice $\Z^d\times \{1,...,N\}$.
If $d\geq 3$, then the chain is transient.
If $d\leq 2$, then the chain is transient if and only if it is non-centered.
\end{theorem}

We now give more details on the proofs of these results.
In the $\mathrm{CAT}(0)$ boundary of the free product $\Z^{d_1}\star \Z^{d_2}$, there are essentially two types of points, namely points in $\Z^{d_i}$ factors and infinite words.

\begin{itemize}
    \item To deal with the first type, we use Proposition~\ref{latticeMartinboundary}.
As announced above, we restrict the random walk to some neighbourhood of the $\Z^{d_i}$ factor and get a $\Z^{d_i}$-invariant Markov chain on $\Z^{d_i}\times \{1,...,N\}$, for some $N$ that depends on the thickening.
However, the random walk on the free product being transient, the induced Markov chain is not a probability.
This is the reason for studying non-probability transition kernels in Propositions~\ref{theorem1'spitzer}, \ref{theorem2spitzer} and \ref{latticeMartinboundary}.
    \item To deal with the second type, we do not need Proposition~\ref{latticeMartinboundary}, but we use transitional sets.
Roughly speaking, since the random walk on the free product is finitely supported, if $y_n$ is converging to an infinite word, to go from $x$ to $y_n$, one has to go through particular sets.
Those are some neighbourhoods of the points along a geodesic from $x$ to $y_n$ where one goes from a $\Z^{d_1}$ coset to a $\Z^{d_2}$ coset or conversely.
Since $x^{-1}y_n$ is converging to an infinite word, there are a lot of transitional sets.
Actually, the number of those tends to infinity.
Using contraction properties of positive matrices, we show that forcing the paths to go through an increasing number of transitional sets implies convergence of the Martin kernel.
\end{itemize}

Let us now say a few words about Propositions~\ref{theorem1'spitzer}, \ref{theorem2spitzer} and \ref{latticeMartinboundary}.
They are generalizations of theorems of P.~Ney and F.~Spitzer in \cite{NeySpitzer}.
The third one, which is the actual description of the Martin boundary is deduced from the second one, which is an asymptotic for the Green function.
Using some technical lemma, proved in \cite{NeySpitzer} (see Lemma~\ref{lemmaNeySpitzer} below), this asymptotic is deduced from the asymptotic of the $n$th convolution power, that is Proposition~\ref{theorem1'spitzer}.
This is a local limit theorem with error terms, which is usually proved for random walks on $\Z^d$ using Fourier theory.
We thus have to generalize classical Fourier theory on $\Z^d$ to $\Z^d$-invariant transition kernels on $\Z^d\times \{1,...,N\}$.
This is precisely what is done in Section~\ref{SectionSub-Markovchainsonathickenedlattice}.

These three results, along with Proposition~\ref{TCL}, are generalizations of the work of W.~Woess on generalized lattices.
As stated above, in \cite{Woess1}, he proves an analog of the law of large numbers for $\Z^d$-invariant transition kernels on $\Z^d\times \{1,...,N\}$.
Our proposition~\ref{TCL} basically states convergence of characteristic functions $\psi^n(\frac{\xi}{\sqrt{n}})$ to some function $\mathrm{e}^{-Q(\xi)/2}$, where $Q$ is a positive definite quadratic form. We deduce from it a central limit theorem (see Theorem~\ref{TCL2}).
Our techniques to prove Proposition~\ref{TCL} are similar to the techniques in Section~8 of \cite{Woess1}, although the formula for the quadratic form in \cite[Proposition~8.20]{Woess1} is wrong.

\subsection{Markov chains and Martin boundary}\label{SectionMarkovchainsandMartinboundary}
Let us give a proper definition of the Martin boundary.
We will only deal with random walks and homogeneous chains, but it can be defined in a more general setting.
Consider a countable space $E$
and give $E$ the discrete topology.
Fix some base point $x_0$ in $E$.
Consider now a chain on $E$, defined by a transition kernel $p$.
Denote by $G$ the Green function of the transition kernel.
Recall that
$$G(x,y)=\sum_{n\geq 0}p^{(n)}(x,y),$$
where $p^{(n)}$ is the $n$th convolution power of $p$, that is,
$$p^{(n)}(x,y)=\sum_{x_1,...,x_{n-1}}p(x,x_1)p(x_1,x_2)\cdots p(x_{n-1},y).$$
We do not assume that $p$ has total mass 1, so it does not define an actual Markov chain.
However, we assume that $p$ has finite mass, i.e.\
$$\forall x\in E, \sum_{y\in E}p(x,y)<+\infty.$$

We will always assume that the chain is irreducible, meaning that
for every $x,y\in E$, there exists $n$ such that $p^{(n)}(x,y)>0$.
For a Markov chain, this means that one can go from any $x\in E$ to any $y\in E$ with positive probability.
In this setting, the Green function $G(x,y)$ is closely related to the probability to go from $x$ to $y$.
Denote by $\mathds{P}(x\rightarrow y)$ this probability.
Then, $G(x,y)=\mathds{P}(x\rightarrow y)G(y,y)$.
We will also assume that the chain is transient, meaning that
the Green function is everywhere finite.
For a Markov chain, this just means that almost surely, one can go back to $x$ starting at $x$ only a finite number of times.

Define the Martin kernel
$$K(x,y)=\frac{G(x,y)}{G(x_0,y)}.$$
The Martin compactification of $E$ with respect to $p$ and $x_0$ is the smallest compact space $M$ in which $E$ is open and dense
and such that $K$ can be continuously extended to the space $E\times M$.
The Martin boundary is then defined as
$$\partial M=M\setminus E.$$
The Martin compactification does not depend on the base point $x_0$ up to isomorphism.

\medskip
To actually show that a compact space $M$ is the Martin compactification of $E$, one has to check the following:
\begin{enumerate}
    \item $E$ is open and dense in $M$.
    \item If $y_n$ converges to $\tilde{y}$ in $M$, then $K(\cdot , y_n)$ converges pointwise to $K(\cdot , \tilde{y})$.
    Since $E$ is discrete, it is sufficient to deal with $\tilde{y}$ in the boundary $\partial M$.
    \item If $\tilde{y}_1\neq \tilde{y}_2$ in $M$, then $K(\cdot, \tilde{y}_1)\neq K(\cdot, \tilde{y}_2)$.
    Again, it is sufficient to deal with points in the boundary.
\end{enumerate}
For more details, we refer to the survey \cite{Sawyer} of S.~Sawyer. There are a lot of examples in there and a full (abstract) construction of the Martin boundary.

In the particular case of a symmetric Markov chain, that is a Markov chain satisfying $p(x,y)=p(y,x)$,
the Green distance, which was defined by S.~Brofferio and S.~Blachère in \cite{BrofferioBlachere} as
$$d_G(x,g)=-\mathrm{ln}\mathds{P}(x\rightarrow y),$$
is actually a distance
and the Martin compactification of $E$, with respect to the Markov chain $p$ is
the horofunction compactification of $E$ for the Green distance.

\medskip
One important aspect of the Martin boundary is its relation with harmonic functions.
Recall that if $p$ is a transition kernel on a countable space $E$, a harmonic function is a function $\phi:E \rightarrow \R$ such that $p\phi=\phi$, that is,
$$\forall x \in E, \phi(x)=\sum_{y\in E}p(x,y)\phi(y).$$
We have the following key property (see \cite[Theorem~4.1]{Sawyer}).
\begin{prop}
Let $p$ be a transient and irreducible transition kernel on a countable space $E$.
For any non-negative harmonic function $\phi$, there exists a measure $\mu_{\phi}$ on the Martin boundary $\partial E$ of $E$ such that
$$\forall x\in E, \phi(x)=\int_{\partial E}K(x,\tilde{x})\mathrm{d}\mu_{\phi}(\tilde{x}),$$
where $K(\cdot, \cdot)$ is the Martin kernel associated to $p$.
\end{prop}

Let $\phi$ be a non-negative harmonic function. It is called minimal if whenever $\psi$ is another non-negative harmonic function such that $\psi(x)\leq \phi(x)$ for every $x\in E$,
then $\psi$ is proportional to $\phi$.
The minimal Martin boundary is the set
$$\partial_mE=\{\tilde{x}\in \partial E,K(\cdot, \tilde{x}) \text{ is minimal harmonic}\}.$$
It is thus a subset of the full Martin boundary $\partial E$.
A classical representation theorem of G.~Choquet shows that for any non-negative harmonic function $\phi$, one can choose the support of the measure $\mu_{\phi}$ lying in $\partial_mE$
(see the first section of \cite{Sawyer}).
In other words, for any such function $\phi$, there exists a unique measure $\mu_{\phi}$ on $\partial_mE$ such that
$$\forall x\in E, \phi(x)=\int_{\partial_mE}K(x,\tilde{x})\mathrm{d}\mu_{\phi}(\tilde{x}).$$

In many situations, the minimal Martin boundary coincides with the full Martin boundary, but it can be a proper subspace, even for (non-finitely supported) random walks in $\Z$ (see \cite{CartwrightSawyer}).

\subsection{Geometric $\mathrm{CAT}(0)$ boundaries}\label{SectionGeometricboundaries}
We will identify Martin boundaries for chains on $\Z^d\times \{1,...,N\}$ and on $\Z^{d_1}\star \Z^{d_2}$ with geometric boundaries.
We now give proper definitions of these.

Consider the map $x\in \Z^d\mapsto \frac{x}{1+\|x\|}$ that embeds $\Z^d$ into the unit ball.
Then take the closure of this embedding.
This is by definition the $\mathrm{CAT}(0)$ compactification of $\Z^d$.
P.~Ney and F.~Spitzer showed that it is also the Martin boundary of a non-centered random walk on $\Z^d$.
Denote by $\partial \Z^d$ the boundary $\overline{\Z^d}\setminus \Z^d$.
It is a sphere at infinity.
Precisely, a sequence $(y_n)$ in $\Z^d$ converges to some point in the boundary if $\|y_n\|$ tends to infinity and if $\frac{y_n}{\|y_n\|}$ converges to some point on the unit sphere.

More generally, we define the $\mathrm{CAT}(0)$ boundary of $\Z^d\times \{1,...,N\}$ forgetting the thickening $\{1,...,N\}$ at infinity.
More accurately, a sequence $(y_n,k_n)$ converges to some point in the boundary if $\|y_n\|$ tends to infinity and if $\frac{y_n}{\|y_n\|}$ converges to some point on the unit sphere.
We denote by $\partial (\Z^d\times \{1,...,N\})$ the boundary thus defined and call $ \Z^d\times \{1,...,N\}\cup \partial( \Z^d\times \{1,...,N\})$ the $\mathrm{CAT}(0)$ compactification of $\Z^d\times \{1,...,N\}$.

Since a finitely generated abelian group $\Gamma$ is of the form $\Z^d\times L$, where $L$ is finite, we can also define its $\mathrm{CAT}(0)$ boundary.
Precisely, identifying $L$ with some set $\{1,...,N\}$, we define the Martin boundary of $\Gamma$ as the Martin boundary of $\Z^d\times \{1,...,N\}$.
The same works with a finitely generated virtually abelian group $\Gamma$.
Indeed, for such a group, there is a subgroup isomorphic to $\Z^d$ and with finite index.
Denote by $L$ the quotient, which is a finite set.
Any section $L\rightarrow \Gamma$ provides an identification between $\Gamma$ and a set $\Z^d\times \{1,...,N\}$
and one can define the $\mathrm{CAT}(0)$ boundary of $\Gamma$ as the $\mathrm{CAT}(0)$ boundary of $\Z^d\times \{1,...,N\}$, since it neither depends on the choice of the abelian subgroup, nor on the choice of the section $L\rightarrow \Gamma$.

Now, let us deal with free products.
Consider the group $\Z^{d_1}\star \Z^{d_2}$ and denote by $e$ the neutral element.
An element of $\Z^{d_1}\star \Z^{d_2}$ that differs from $e$ can be written as $a_1b_1...a_nb_n$ with $a_i\in \Z^{d_1}$ and $b_i\in \Z^{d_2}$.
Furthermore, this writing is unique if $a_i\neq 0$ except maybe $a_1$ and if $b_i\neq 0$, except maybe $b_n$.
Thus, an element of $\Z^{d_1}\star \Z^{d_2}$ can be represented by a finite sequence, alternating elements of $\Z^{d_1}$ and elements of $\Z^{d_2}$.
An infinite word is an infinite sequence $a_1b_1...a_nb_n...$ alternating elements of $\Z^{d_1}$ and elements of $\Z^{d_2}$.

The prefix of size $p$ of a word, finite or infinite, is the sub-sequence of its $p$ first elements.
A sequence $(g_k)$ in $\Z^{d_1}\star \Z^{d_2}$ converges to some infinite word $\tilde{g}$ if for every $p$, there exists $k_0$ such that for every $k\geq k_0$, the prefixes of size $p$ of $g_k$ and $\tilde{g}$ exist and are the same.
Let $a_1b_1...a_nb_n\in \Z^{d_1}\star \Z^{d_2}$, with $b_n\neq 0$.
Consider a sequence $(g_k)$ such that there exists $k_0$ such that for every $k\geq k_0$,
$g_k$ is of the form $a_1b_1...a_nb_na_{n+1,k}...$ and $a_{n+1,k}\in \Z^{d_1}$ converges to some point in the $\mathrm{CAT}(0)$ boundary of $\Z^{d_1}$,
that is $\|a_{n+1,k}\|$ tends to infinity and $\frac{a_{n+1,k}}{\|a_{n+1,k}\|}$ converges to some point in the unit sphere.
We say that $(g_k)$ converges in the boundary in the $\Z^{d_1}$ factor $a_1b_1...a_nb_n\Z^{d_1}$.
Similarly, we can define convergence in the boundary in $\Z^{d_2}$ factors.
The $\mathrm{CAT}(0)$ compactification of $\Z^{d_1}\star \Z^{d_2}$ is obtained by gluing together infinite words with infinite spheres in each $\Z^{d_i}$ factor.
In other words, a sequence $(g_k)$ converges to some point in the $\mathrm{CAT}(0)$ boundary if $(g_k)$ converges to some infinite word, or if $(g_k)$ converges in the boundary in some $\Z^{d_i}$ factor.
Similarly, we can define the $\mathrm{CAT}(0)$ boundary of a free product $\Gamma_1\star \Gamma_2$, where $\Gamma_1$ and $\Gamma_2$ are finitely generated virtually abelian groups.

Those spaces we defined are compact and $\Z^d$, $\Z^d\times \{1,...,N\}$ and $\Z^{d_1}\star \Z^{d_2}$ are dense in their $\mathrm{CAT}(0)$ compactification and the same holds with $\Z^d$ and $\Z^{d_i}$ replaced with finitely generated virtually abelian groups.

Why $\mathrm{CAT}(0)$?
In all the spaces we considered above, replacing $\Z^d$ by $\R^d$ provides a space with negative curvature in some weak sense, precisely defined in \cite{BridsonHaefliger}.
Those spaces are called $\mathrm{CAT}(0)$ spaces because their geometry can be compared with Euclidean geometry.
In those spaces, there already is a notion of boundary at infinity, called the visual boundary.
For example in trees, we can define leaves at infinity, whereas a (simply connected) negatively curved manifold is diffeomorphic to the unit ball of $\R^n$ and thus has a geometric boundary which is the unit sphere.
In those two examples, the boundary is the Gromov boundary.
The boundaries we defined for $\Z^d$ and $\Z^{d_1}\star \Z^{d_2}$ are actually the $\mathrm{CAT}(0)$ boundaries of $\R^d$ and $\R^{d_1}\star \R^{d_2}$.
Since $\Z^d$ co-compactly acts on $\R^d$, gluing spheres at infinity in our context has a geometric meaning.
When considering finite extensions $\Z^d\times \{1,...,N\}$, we keep the name $\mathrm{CAT}(0)$ compactification for simplicity, although we do not get $\mathrm{CAT}(0)$ spaces per se.

\begin{rem}
Notice that the boundary we defined for $\Z^{d_1}\star \Z^{d_2}$ coincides with the boundary defined by F.~Dahmani in \cite[Section~3]{Dahmani}, seeing $\Z^{d_1}\star \Z^{d_2}$ as a hyperbolic group relative to the peripheral subgroups $\Z^{d_1}$ and $\Z^{d_2}$ and choosing the $\mathrm{CAT}(0)$ boundaries for these subgroups.
This alternative definition of geometric boundary could be useful in another context.
\end{rem}

\medskip
The paper is organized as follows.
In Section~\ref{SectionTwotheoremsfromlinearalgebra}, we give two results from linear algebra that will be repeatedly used during the proofs of all our theorems.
In Section~\ref{SectionSub-Markovchainsonathickenedlattice}, we study sub-Markov and Markov chains on $\Z^d\times \{1,...,N\}$ and prove every technical results we need in the following for such chains. The assumptions are given so that the results can be used in several contexts.
In Section~\ref{SectionMarkovchainsonathickenedlattice}, we show that the assumptions of Section~\ref{SectionSub-Markovchainsonathickenedlattice} are satisfied for non-centered Markov chains on $\Z^d\times \{1,...,N\}$ and we deduce from that the identification of the Martin boundary stated in Theorem~\ref{theorem2}. We also prove there Theorem~\ref{theorem3}.
In Section~\ref{SectionRandomwalksinfreeproducts}, we prove the first part of Theorem~\ref{theorem1}, i.e.\ we prove that the Martin boundary coincides with the geometric boundary in free products, again using results of Section~\ref{SectionSub-Markovchainsonathickenedlattice}.
Finally, in Section~\ref{SectionMinimalMartinboundary}, we deal with the minimal Martin boundary, ending the proofs of Theorems~\ref{theorem1} and~\ref{theorem2}.

I would like to warmly thank Sebastien Gouëzel for many helpful conversations and comments on this paper.


\section{Two theorems from linear algebra}\label{SectionTwotheoremsfromlinearalgebra}
In this section, we show two useful results, namely a strong version of the Perron-Frobenius theorem and a spectral perturbation theorem.
The reason for using these results is that in the following section, we will have to deal with generalized Fourier transforms.
Throughout the proofs, we will find matrices with a dominant eigenvalue. We will also have to study small perturbations of these dominant eigenvalues.

\subsection{The Perron-Frobenius theorem}\label{SectionPerronFrobenius}
The setting here is the $k$-dimensional vector space $\R^k$.
Let $T$ be a square matrix of size $k$ with non-negative entries.
Let $T^n$ denote the $n$th power of $T$ and let $t_{i,j}$ be the entries of $T$ and $t_{i,j}^{(n)}$ the entries of $T^n$.
The matrix $T$ is said to be \emph{irreducible} if for every $i,j$, there exists $n$ such that $t_{i,j}^{(n)}$ is positive and
it is said to be \emph{strongly irreducible} if there exists $n$ such that every entry of $T^n$ is positive.

The classical Perron-Frobenius theorem is the following (see \cite[Theorem~1.1]{Seneta}).
\begin{theorem}\label{PerronFrob1}
Let $T$ be a strongly irreducible matrix.
Then, $T$ has an eigenvector $v_T$ with positive entries.
Let $\lambda_T$ be the associated eigenvalue.
Then, $\lambda_T$ is positive and simple.
Furthermore, for any other eigenvalue $\lambda$ of $T$, $\lambda_T>|\lambda|$.
Finally, up to scaling, $v_T$ is the unique non-negative eigenvector.
\end{theorem}

One particular example of such a matrix is a \emph{stochastic matrix}, that is a matrix $T$ such that
$$\forall i,\sum_{j}t_{i,j}=1.$$
In that case, $\lambda_T=1$ and one can take $v_T$ to be $v_T=(1,...,1)$.
One says that $T$ is \emph{sub-stochastic} if
$$\forall i,\sum_{j}t_{i,j}\leq 1$$
and \emph{strictly sub-stochastic} if at least one of these inequalities is strict.
If $T$ is strictly sub-stochastic and strongly irreducible, then $\lambda_T<1$.

There are a lot of proofs for this theorem.
One of them makes use of contraction properties for the Hilbert distance.
It was published by G.~Birkhoff in \cite{Birkhoff}
and generalized by Y.~Derrienic in \cite[Lemma~IV.4]{Derriennic}.
We now give a more general version of this and begin with some definitions about the Hilbert distance.
We take the same notations as in \cite{Derriennic}.

Let $C$ be the cone of vectors $v$ in $\R^k$ with positive coordinates and $C'$ be the intersection of $C$ with the unit sphere:
$C'=C\cap \Ss^{k-1}$.
If $x\neq y\in C'$, the lines $Dx=(Ox)$ and $Dy=(Oy)$ they generate are distinct.
Let $D$ and $D'$ be the lines of $\R^k$ obtained as the intersection of the plane that $Dx$ and $Dy$ generate and the boundary of $C$ in $\R^k$.
One then has four distinct lines.

Let $\Delta_k:=\{v=(v_1,...,v_k)\in \R^k,v_i\geq 0,v_1+...+v_k=1\}$ be the $k-1$-dimensional standard regular simplex.
The four lines $D$, $Dx$, $Dy$ and $D'$ intersect $\Delta_k$ in four distinct aligned points.
Up to changing $D$ and $D'$, one can assume that they are aligned in this order.
By choosing some affine coordinate system on that line, one gets four coordinates $u$, $\tilde{x}$, $\tilde{y}$ and $v$.
The \emph{cross-ratio} of $D$, $Dx$, $Dy$ and $D'$ is by definition
$$[Dx,Dy,D,D']=\frac{(v-\tilde{x})(u-\tilde{y})}{(u-\tilde{x})(v-\tilde{y})}.$$
The \emph{Hilbert distance} between $x$ and $y$ is then
$$d_{\mathcal{H}}(x,y)=\frac{1}{2}\mathrm{log}[Dx,Dy,D,D'].$$

\begin{rem}
The cross-ratio does not depend on the choice of affine coordinates and $d_{\mathcal{H}}$ is a distance.
We refer to \cite{Birkhoff} for more details.
\end{rem}

\begin{rem}\label{remarktopologies}
Obviously the Hilbert distance differs from the Euclidean one, since it gives $C'$ an infinite diameter.
Nevertheless, the induced topologies are the same.
In particular, a compact set for one of the distances is also compact for the other one.
\end{rem}

Let $T$ be a $k\times k$ matrix and assume that $T$ has non-negative entries.
Also assume that the zeros of $T$ are divided into columns, that is either a column is entirely null, either it only has positive entries.
Such a matrix, if it is not the null matrix, acts on $C'$ via the map
$$T: v\in C' \mapsto T\cdot v=\frac{Tv}{\|Tv\|}\in C'.$$
Indeed, if $v\in C'$, then every coordinate of $v$ is positive and since the zeros of $T$ are divided into columns, there is at least one positive entry on each line of $T$, so that every coordinate of $Tv$ is positive.
When considering a matrix whose zeros are divided into columns, we will always implicitly assume that it is not the null matrix.

\begin{definition}
The diameter of a non-negative $k\times k$ matrix $T$, whose zeros are divided into columns, is the diameter of $\{T\cdot v,v\in C'\}$ for the distance $d_{\mathcal{H}}$.
It is denoted by $\Delta(T)$.
\end{definition}

\begin{lem}\label{lemmacontraction}
Let $T$ be a non-negative $k\times k$ matrix and assume that its zeros are divided into columns.
Assume that $\Delta(T)$ is finite.
Then, for every $x,y\in C'$, one has
$d_{\mathcal{H}}(T\cdot x,T\cdot y)\leq \delta(T)d_{\mathcal{H}}(x,y),$
with $\delta(T)=\mathrm{th}(\frac{1}{4}\Delta(T))$.
\end{lem}
\begin{proof}
This is exactly \cite[Lemma~1]{Birkhoff}.
\end{proof}

The following theorem generalizes \cite[Lemma~4]{Derriennic} and the proof is the same
but we provide it for the convenience of the reader.

\begin{theorem}\label{PerronFrob2}
Let $(T_i)_{i\in I}$ be a family of non-negative $k\times k$ matrices and assume that their zeros are divided into columns.
Also assume that the diameters $\Delta(T_i)$ are uniformly bounded by some real number $\Delta\geq 0$.
Let $i_1,...,i_n,...$ be a sequence of indices and let $\tilde{T}_n=T_{i_1}...T_{i_n}$.
Then, for every $x\in C'$, $\tilde{T}_n\cdot x$ converges to some vector $\tilde{x}\in C'$.
Furthermore, the convergence is uniform on $C'$.
The limit vector $\tilde{x}$ depends on the sequence $i_1,...,i_n...$ but does not depend on the vector $x$.
\end{theorem}

\begin{proof}
Let $\delta=\mathrm{th}(\frac{1}{4}\Delta)$.
According to Lemma~\ref{lemmacontraction},
for every $i\in I$ and for every $x,y\in C'$,
$d_{\mathcal{H}}(T\cdot x,T\cdot y)\leq \delta d_{\mathcal{H}}(x,y)$.
Thus, the diameter of the range of $\tilde{T}_n$ is bounded by $\delta^{n-1}\Delta$.
The ranges of $\tilde{T}_n$ form a non-increasing sequence of relatively compact sets whose diameters converge to zero.
Thus, there exists a unique $\tilde{x}$ in the intersection of the closures of these ranges.
It implies uniform convergence of $\tilde{T}^n\cdot x$ to $\tilde{x}$.
\end{proof}

\subsection{Spectral perturbation theorem}\label{SectionSpectralpertubationtheorem}
As announced, in the following we will need a result about perturbations of linear operators.
Let $T_0$ be a $k\times k$ matrix with a dominant eigenvalue, that is an eigenvalue $\lambda_0$ that is simple and such that for every other eigenvalue $\lambda$ of $T_0$, $|\lambda_0|>|\lambda|$.
Recall that mapping a matrix $T$ to its eigenvalues is continuous, so that if $T$ is close to $T_0$, then $T$ also has a dominant eigenvalue $\lambda_T$.
Let $E_0$ and $E_T$ be the eigenspaces associated to $\lambda_0$ and $\lambda_T$ and let $\pi_0$ and $\pi_T$ be the spectral projectors onto $E_0$ and $E_T$.
By definition, $\pi_0=v_0\cdot \nu_0$, where $\nu_0$ is a left eigenvector for $T_0$ associated to $\lambda_0$ and $v_0$ is a right eigenvector.
Seeing $\nu_0$ as a line vector and $v_0$ as a column vector, the product $v_0\cdot \nu_0$ is indeed a $k\times k$ matrix.
Similarly, $\pi_T=v_T\cdot \nu_T$.

\begin{theorem}\label{theoremKato}
With these notations, there exists a neighborhood $\mathcal{U}$ of $T_0$ such that if $T\in \mathcal{U}$, then
$$T=\lambda_T\pi_T+R_T,$$
where $R_T$ is some remainder.
Furthermore, $R_T\pi_T=\pi_TR_T=0$.
In particular,
$$T^n=\lambda_T^n\pi_T+R_T^n,$$
with $\|R_T^n\|\leq C\tilde{\lambda}^n$
where $C$ and $\tilde{\lambda}$ do not depend on $T$ and with $\tilde{\lambda}<\lambda_0$.
\end{theorem}

This follows from \cite[Theorem~III.6.17]{Kato}, associated to the decomposition of $\R^k$ into
the direct sum $\R^k=E_T\oplus F_T$, where $F_T$ is the range of $I_k-\pi_T$.

All the quantities involved above not only depend continuously, but also analytically on $T$.
Precisely, we have the following.

\begin{prop}\label{propanalycity}
With these notations, the maps $T\mapsto \lambda_T$ and $T\mapsto \pi_T$ are analytic in $T\in \mathcal{U}$.
\end{prop}

This is a consequence of \cite[Theorem~VII.1.8]{Kato}.
One can also use the implicit function theorem, as it is done in \cite{Woess1} (see Proposition~8.20 there).


\section{Sub-Markov chains on a thickened lattice}\label{SectionSub-Markovchainsonathickenedlattice}

\subsection{Transition kernel on a thickened lattice}\label{Sectiontransitionkernelonathickenedlattice}
We consider here a chain on the set $\Z^d\times \{1,...,N\}$ defined by the transition kernel
$$p((x,k),(y,j))=p_{k,j}(x,y)\geq 0, x,y\in \Z^d, k,j\in \{1,...,N\}.$$
We will later apply our results to sub-Markov chains, that is transition kernels satisfying
$$\forall x\in \Z^d,\forall k\in \{1,...,N\}, \sum_{y\in \Z^d}\sum_{1\leq j\leq N}p_{k,j}(x,y)\leq 1.$$
However, during the proofs, we will have to use $h$-processes, meaning that we will modify the transition kernel $p$ using harmonic functions.
The new chains that we will get will not be sub-Markov.
Thus, we will need statements without such an assumption, but one should keep in mind that the chains we want to study are indeed sub-Markov.

We will always assume that our chain is invariant under the action of $\Z^d$ on $\Z^d\times \{1,...,N\}$, that is
$$\forall x,y\in \Z^d, \forall k,j\in \{1,...,N\}, p_{k,j}(x,y)=p_{k,j}(0,y-x).$$

We will denote by $p^{(n)}$ the $n$th convolution power of the chain.
Recall that
$$p^{(n)}((x,k),(y,j))=\underset{...,(x_{n-1},k_{n-1})}{\underset{(x_1,k_1),...,}{\sum}}p_{k,k_1}(x,x_1)p_{k_1,k_2}(x_1,x_2)\cdots p_{k_{n-1},j}(x_{n-1},y).$$
The $\Z^d$-invariance is preserved:
$$p^{(n)}((x,k),(y,j))=p^{(n)}((0,k),(y-x,j))=:p^{(n)}_{k,j}(0,y-x).$$

\begin{definition}\label{Defirred}
We say that the chain is irreducible if for every $x,y\in \Z^d$ and for every $k,j$, there exists $n$ such that $p^{(n)}_{k,j}(x,y)>0$.
We say that the chain is strongly irreducible if for every $k,j\in \{1,...,N\}$ and for every $x,y\in \Z^d$, there exists $n_0$ such that for every $n\geq n_0$, $p^{(n)}_{k,j}(x,y)>0$.
\end{definition}
We now assume that the chain is strongly irreducible.
It will allow us later to deal with strongly irreducible matrices and to use the Perron-Frobenius theorem (Theorem~\ref{PerronFrob1}).
The assumption that we really want to make is (weak) irreducibility.
However, it turns out that to understand the Martin boundary, one can assume strong irreducibility (see Lemma~\ref{Spitzertrick} below).

If the chain is strongly irreducible, it is in particular irreducible.
Actually, the following condition is satisfied: for every $x,y\in \Z^d$, for every $k,j\in \{1,...,N\}$, there exists $n\in \N^*$ such that $p^{(n)}_{k,j}(0,x)>0$ and $p^{(n)}_{k,j}(0,y)>0$.
We will use this in the proof of our local limit theorem (see Proposition~\ref{theorem1spitzer}).

It may happen that at each level $k$, the sum $\sum_{x,j}p_{k,j}(0,x)$ equals 1.
In this case, the chain is Markov.
If it is sub-Markov and if at least for one of the level $k$, this sum is smaller than 1, then the chain is transient.
In that case, the Green function is finite.
Recall that
$$G((x,k)(y,j))=\sum_{n\in \N}p^{(n)}_{k,j}(x,y)=\sum_{n\in \N}p^{(n)}_{k,j}(0,y-x)=:G_{k,j}(x,y),$$
where $p^{(0)}(x,y)=1$ if $x= y$ and 0 otherwise.
In any case, one can define the Green function but it can take the value $+\infty$.

Our goal is to understand the Martin kernel associated to the transition kernel $p$.
For a random walk on the lattice $\Z^d$, (minimal) harmonic functions are of the form $x\in \Z^d\mapsto C\mathrm{e}^{u\cdot x}$, with $u$ satisfying the condition $\sum_{x\in \Z^d}p(0,x)\mathrm{e}^{u\cdot x}=1$.
Actual harmonic functions are linear combinations of those.
These harmonic functions are useful to understand the Martin kernel.
P.~Ney and F.~Spitzer study them to fully describe the Martin boundary in \cite{NeySpitzer}.
It turns out, as we will see, that on the thickened lattice $\Z^d\times \{1,...,N\}$, harmonic functions have a similar form.
Thus, for $u\in \R^d$, on defines the modified transition kernel
$$p_{k,j;u}(0,x)=p_{k,j}(0,x)\mathrm{e}^{u\cdot x},\text{ } p_{k,j;u}(x,y)=p_{k,j;u}(0,y-x).$$
The new chain is again strongly irreducible.
Furthermore, if the first one has finite support, so does the new one.
We will also denote by $P_u(0,x)$ the associated matrix, with entries $p_{k,j;u}(0,x)$ and $G_{k,j;u}$ the associated Green function.
One can easily check that $p^{(n)}_{k,j;u}(x,y)=p^{(n)}_{k,j}(x,y)\mathrm{e}^{u\cdot(y-x)}$,
so that $G_{k,j;u}(x,y)=G_{k,j}(x,y)\mathrm{e}^{u\cdot (y-x)}$.
Thus, the new Green function is finite if and only if the first one is.

We now define a matrix which plays a fundamental role in the following.
Let $F(u)$ be the $N\times N$ matrix with entries
$$F_{k,j}(u)=\sum_{x\in \Z^d}p_{k,j;u}(0,x)=\sum_{x\in \Z^d}p_{k,j}(0,x)\mathrm{e}^{u\cdot x}.$$

A priori, $F_{k,j}(u)$ takes values in $[0,\infty]$.
Denote by $\mathcal{F}$ the interior of the set of $u$ such that $F_{k,j}(u)<\infty$ for all $k$ and $j$.
The set $\mathcal{F}$ is convex since the exponential function is convex,
but it depends on the tail of the chain.
We will later impose some condition on the tail to ensure $F$ is finite on a sufficiently large set (see Assumption~\ref{hyp1}).

\begin{lem}
Let $F(u)^n$ be the $n$th power of $F(u)$.
The entries of $F(u)^n$ are given by
$$F_{k,j}(u)^n=\sum_{x\in \Z^d}p^{(n)}_{k,j;u}(0,x)=\sum_{x\in \Z^d}p^{(n)}_{k,j}(0,x)\mathrm{e}^{u\cdot x}.$$
\end{lem}
\begin{proof}
By definition,
\begin{align*}
    F_{k,j}(u)^2&=\sum_{x,y\in \Z^d}\sum_{l\in \{1,...,N\}}p_{k,l}(0,x)p_{l,j}(0,y)\mathrm{e}^{u\cdot(x+y)}\\
    &=\sum_{x,y\in \Z^d}\sum_{l\in \{1,...,N\}}p_{k,l}(0,x)p_{l,j}(x,x+y)\mathrm{e}^{u\cdot(x+y)}.
\end{align*}

The change of variables $y=z-x$ gives
$$F_{k,j}(u)^2=\sum_{x,z\in \Z^d}\sum_{l\in \{1,...,N\}}p_{k,l}(0,x)p_{l,j}(x,z)\mathrm{e}^{u\cdot z}=\sum_{z\in \Z^d}p^{(2)}_{k,j}(0,z)\mathrm{e}^{u\cdot z},$$
which is our formula for $n=2$.
We conclude by an induction argument.
\end{proof}

Thus, $F(u)$ is a non-negative matrix and it is strongly irreducible, meaning that there exists $n$ such that every entry of $F(u)^n$ is positive.
According to Theorem~\ref{PerronFrob1}, whenever $F(u)<\infty$, it admits a unique dominant eigenvalue, which is real and simple.
We denote it by $\lambda(u)$.
It is actually this eigenvalue that will play the role of the functions 
$u\mapsto \sum_{x\in \Z^d}p(0,x)\mathrm{e}^{u\cdot x}$
studied by P.~Ney and F.~Spitzer in \cite{NeySpitzer}.
Thus, to ensure that this eigenvalue is well defined, we have to assume strong irreducibility instead of weak irreducibility.

In \cite{NeySpitzer}, the authors focus on the set where those functions take value 1.
By analogy, we introduce the set where $\lambda(u)=1$.
We denote this set by $H$:
$$H=\{u\in \mathcal{F}, \lambda(u)=1\}.$$

In the following, we will define some functions depending on $u$.
It will be implicit that they are defined on the set of $u\in \R^d$ such that $F(u)$ is finite.
Denote by $C(u)$ a right eigenvector and $\nu(u)$ a left eigenvector for $F(u)$, both associated to the eigenvalue $\lambda(u)$.
Since $F$ is strongly irreducible, one can choose these eigenvectors to have positive coordinates.
We impose the following normalization.
First choose $\nu(0)$ so that $\sum \nu(0)_j=1$, then impose that $\nu(0)\cdot C(u)=1$ and
finally, impose that $\nu(u)\cdot C(u)=1$ for every $u$.
Seeing $\nu(u)$ as a line vector and $C(u)$ as a column vector, $C(u)\cdot \nu(u)$ is an $N\times N$ matrix.
It is actually the spectral projector onto the eigenspace associated to $\lambda(u)$.
We will denote it by $\pi(u)$.

\begin{lem}
With these notations, the functions $\nu$, $C$ and $\lambda$ are analytic in $u\in \mathcal{F}$.
\end{lem}

\begin{proof}
Proposition~\ref{propanalycity} states that $\lambda$ is analytic in $F$.
Since $F$ is itself analytic in $u$, $\lambda$ is analytic in $u$.
Furthermore, $\pi(u)$ is analytic in $u$ and
$$\nu(u)=\nu(0)\cdot C(u)\cdot \nu(u)=\nu(0)\pi(u)$$
and
$$C(u)=\frac{1}{\nu(u)\cdot C(0)}C(u)\cdot \nu(u)\cdot C(0)=\frac{1}{\nu(u)\cdot C(0)}\pi(u)C(0),$$
so that $\nu$ and $C$ also are analytic.
\end{proof}

\begin{rem}
The eigenvectors $\nu$ and $C$ are only determined up to multiplication by a constant.
Obviously, one can choose a non-analytic dependency on $u$.
The normalization we chose forces analycity and will be convenient in the proofs below.
\end{rem}

Thanks to this lemma, we are able to differentiate the functions $\nu$, $C$ and $F$.
In the following, we will use the notation $\nabla$ to denote a gradient and $\nabla^2$ to denote a Hessian quadratic form.
Let us say a few words about the type of formulae we will get.
When differentiating in $u$, the gradient $\nabla$ denotes a vector in $\R^d$.
For example, differentiating $C(u)$, one gets a vector of $\R^d$ vectors.
Thus, the notation $\nabla C(u)$ is purely formal and one should understand that every coefficient of $C(u)$ is differentiated.
In the same spirit, when differentiating $F(u)$, one gets a matrix of vectors.
When differentiating a product, for example $\nabla [\nu(u) \cdot C(u)]$, one multiplies $\R^N$ vectors or $N\times N$ matrices and the result is a vector or matrix of $\R^d$ vectors.
Precisely, in this example, one has
$$\nabla [\nu(u) \cdot C(u)]=\sum_{k\in \{1,...,N\}}\nabla [\nu(u)(k)C(u)(k)].$$
Using a similar notation, $\nabla^2F$ is a matrix of quadratic forms.
Precisely, $\nabla F(u)$ is the matrix with entries
$$\nabla F_{k,j}(u)=\sum_{x\in \Z^d}xp_{k,j}(0,x)\mathrm{e}^{u\cdot x}$$
and $\nabla^2F(u)$ is the matrix with entries the quadratic forms
$$\nabla^2 F_{k,j}(u)(\theta)=\sum_{x\in \Z^d}(x\cdot \theta)^2p_{k,j}(0,x)\mathrm{e}^{u\cdot x}, \theta \in \R^d.$$

Since $\nu(u)$ and $C(u)$ are left and right eigenvectors, $\lambda(u)=\nu(u)F(u)C(u)$.
Differentiating this formula, one gets
\begin{align*}
    \nabla \lambda (u)&=\nabla \nu (u) F(u) C(u)+\nu(u) \nabla F(u)C(u)+\nu(u)F(u)\nabla C(u)\\
    &=\lambda(u)(\nabla \nu (u) C(u)+\nu(u)\nabla C(u))+\nu(u)\nabla F(u)C(u).
\end{align*}
Differentiating the equality $\nu(u)\cdot C(u)=1$, one gets $\nabla \nu(u)C(u)+\nu(u)\nabla C(u)=0$.
Thus,
\begin{equation}\label{equationlambdaprime}
    \nabla \lambda (u)=\nu(u)\nabla F(u)C(u).
\end{equation}

\begin{prop}\label{strictconvexity}
If the chain is strongly irreducible, the function $\lambda:u\mapsto \lambda(u)$ is strictly convex on the convex set $\mathcal{F}$.
\end{prop}

\begin{proof}
Everything is deduced from irreducibility of $p$.
One has to show that for every $u$, the Hessian quadratic form $\nabla^2\lambda(u)$ is positive definite.
First, we show we can assume that $u=0$.
Indeed, assume we know that $\nabla^2 \lambda(0)$ is positive definite, fix $u$ and consider the transition kernel $\tilde{p}$, defined as
$$\tilde{p}(0,x)=p(0,x)\mathrm{e}^{u\cdot x}.$$
The associated matrix $\tilde{F}$ has entries
$$\tilde{F}_{k,j}(v)=\sum_{x\in \Z^d}\tilde{p}_{k,j}(0,x)\mathrm{e}^{v\cdot x}=\sum_{x\in \Z^d}p_{k,j}(0,x)\mathrm{e}^{u\cdot x}\mathrm{e}^{v\cdot x}=F_{k,j}(u+v).$$
Thus, the associated eigenvalue $\tilde{\lambda}$ satisfies $\tilde{\lambda}(v)=\lambda(u+v)$.
In particular, differentiating twice, $\nabla^2\tilde{\lambda}(0)=\nabla^2\lambda(u).$
Since $\tilde{p}$ also is irreducible, $\nabla^2\tilde{\lambda}(0)$ is positive definite and so is $\nabla^2\lambda(u)$.

Now, let us show that $\nabla^2\lambda(0)$ is positive definite.
Recall that $\nu(0)\cdot C(u)=1$ and that $\nu(u)\cdot C(u)=1$ for every $u$.
Using Taylor-Young, one has
$$C(u)=C(0)+\nabla C(0)\cdot u+O(u^2).$$
Denote by $C(u)_k$ the coordinates of $C(u)$ and rewrite Taylor-Young for each coordinate:
$$C(u)_k=C(0)_k+\nabla C(0)_k\cdot u + O(u^2).$$
Let $D_u$ be the diagonal matrix
$$D_u=\begin{pmatrix}
\mathrm{e}^{\frac{\nabla C(0)_1\cdot u}{C(0)_1}} & & \\
& \ddots & \\
& & \mathrm{e}^{\frac{\nabla C(0)_N\cdot u}{C(0)_N}}
\end{pmatrix}$$
and let $\widehat{F}(u)$ be the matrix $\widehat{F}(u):=D_u^{-1}F(u)D_u$.
Since it is conjugate to $F(u)$, it has the same eigenvalues.
Besides, $D_u^{-1}C(u)$ is an eigenvector for $\widehat{F}(u)$ associated to $\lambda(u)$.
Notice that
$D_u^{-1}C(u)=C(0)+O(u^2)$ so that $\nu(0)\cdot (D_u^{-1}C(u))=1+O(u^2)$.
Define then $\widehat{C}(u)$ by the formula
$$\widehat{C}(u):=\frac{1}{\nu(0)\cdot (D_u^{-1}C(u))}D_u^{-1}C(u)$$
so that
$$\widehat{C}(u)=\frac{1}{1+O(u^2)}(C(0)+O(u^2))=C(0)+O(u^2).$$
Then, $\nu(0)\cdot \widehat{C}(u)=1$ so that
$$\lambda(u)=\nu(0)\widehat{F}(u)\widehat{C}(u)=\nu(0)\widehat{F}(u)C(0)+\nu(0)\widehat{F}(u)[\widehat{C}(u)-C(0)].$$
Moreover, $\nu(0)F(0)=\lambda(0)\nu(0)$, and so
$$\nu(0)F(0)\widehat{C}(u)=\nu(0)F(0)C(0)=\lambda(0).$$
Thus,
$$\lambda(u)=\nu(0)\widehat{F}(u)C(0)+\nu(0)[\widehat{F}(u)-F(0)][\widehat{C}(u)-C(0)]=\nu(0)\widehat{F}(u)C(0)+O(u^3).$$
Showing that the Hessian quadratic form of $\widehat{\lambda}(u):=\nu(0)\widehat{F}(u)C(0)$ is positive definite at 0 is enough to conclude.
It is simpler, since one multiplies $\widehat{F}(u)$ on the left and on the right with constant vectors.

By definition
$$\widehat{F}_{k,j}(u)=\sum_{x\in \Z^d}p_{k,j}(0,x)\mathrm{e}^{u\cdot x}\mathrm{e}^{u\cdot \left (\frac{\nabla C(0)_j}{C(0)_j}-\frac{\nabla C(0)_k}{C(0)_k}\right )}.$$
Differentiating twice this formula, one gets
\begin{align*}
    \nabla^2 \widehat{F}_{k,j}(u)(\theta)=\sum_{x\in \Z^d}\left [\left (x+\frac{\nabla C(0)_j}{C(0)_j}-\frac{\nabla C(0)_k}{C(0)_k}\right )\cdot \theta\right ]^2&p_{k,j}(0,x)\mathrm{e}^{u\cdot x}\\
    &\mathrm{e}^{u\cdot \left (\frac{\nabla C(0)_j}{C(0)_j}-\frac{\nabla C(0)_k}{C(0)_k}\right )}, \theta \in \R^d.
\end{align*}

Assume that there exists $\theta\neq 0\in \R^d$ such that for every $k,j$, $\nabla^2\widehat{F}_{k,j}(0)(\theta)=0$
and write $E=\R\theta$.
Denote by $e_1,...,e_d$ the canonical basis of $\R^d$.
To simplify formulae, also use the notation $\beta_{k,j}=\frac{\nabla C(0)_j}{C(0)_j}-\frac{\nabla C(0)_k}{C(0)_k}$.
Notice then that for every $k,j,l$, $\beta_{k,k}=0$ and $\beta_{k,j}+\beta_{j,l}=\beta_{k,l}$.
For every $k,j$, $p_{k,j}(0,x)\neq 0$ only happens if $x+\beta_{k,j}\in E^{\perp}$.
Let us fix $k$.
Since the chain is irreducible, there exists a path $(0,k),(x_1,j_1),(x_2,j_2),...,(x_n,j_n),(e_1,k)$ such that $p_{k,j_1}(0,x_1)\neq 0$, $p_{j_m,j_{m+1}}(x_m,x_{m+1})\neq 0$ and $p_{j_n,k}(x_n,e_1)\neq0$.
We deduce from this fact that $x_1+\beta_{k,j_1}\in E^{\perp},x_2-x_1+\beta_{j_1,j_2}\in E^{\perp},...,e_1-x_n+\beta_{j_n,k}\in E^{\perp}$.
Summing all these vectors, one gets that $e_1\in E^{\perp}$.
Similarly, every other vector $e_i$ is in $E^{\perp}$, which is a contradiction.

Thus, for every $\theta\neq 0$, there exist some $k$ and $j$ such that
$\nabla^2\widehat{F}_{k,j}(0)(\theta)>0.$
Averaging this with the positive vectors $\nu(0)$ and $C(0)$, one has
$$\nabla^2\lambda(0)(\theta)=\nu(0)\nabla^2\widehat{F}(0)(\theta)C(0)>0,$$ so that $\lambda$ is positive definite at 0.
\end{proof}

\begin{rem}
We only used weak irreducibility of $p$ in the proof, but to ensure that $\lambda(u)$ is well defined, we did use strong irreducibility.
Actually, we could have defined strong irreducibility only for the $\{1,...,N\}$ part,
meaning that for every $k,j$, there exists $n_0$ such that for every $n\geq n_0$, there exist $x,y$ such that $p^{(n)}_{k,j}(x,y)>0$.
With this assumption, the matrix $F(u)$ is strongly irreducible, hence $\lambda(u)$ is well defined.
However, we will not use such technical distinctions in the following, so we stated everything with strong irreducibility of $p$ as in Definition \ref{Defirred}.
\end{rem}

In the following (precisely in the proof of Lemma~\ref{lemmaTCL}), we will need to consider a transition kernel that is not supported on the lattice $\Z^d$, but on some translated lattice.
In this situation, the dominant eigenvalue is also strictly convex.
Precisely, one has the following.

\begin{prop}\label{strictconvexity2}
Let $\alpha\in \R^d$ be a vector with non-necessarily integer coordinates.
Let $\tilde{F}(v),v\in \R^d$ be an $N\times N$ matrix with entries $\tilde{F}_{k,j}(v)$ of the form
$$\tilde{F}_{k,j}(v)=\sum_{x\in \Z^d}a_{k,j}(x)\mathrm{e}^{v\cdot (x+\alpha)}$$
with $a_{k,j}(x)\geq 0$.
Assume that the chain defined by the entries $a_{k,j}(x)$ is strongly irreducible, meaning that for every $x,y$, for every $k,j$, there exists $n_0$ such that if $n\geq n_0$, then one can find $x=x_1,x_2,...,x_n=y$ and $k=k_0,k_1,...,k_n=j$ with $a_{k_i,k_{i+1}}(x_{i+1}-x_i)>0$.
Then, $\tilde{F}$ is strongly irreducible. Denote by $\tilde{\lambda}$ its dominant eigenvalue.
Then, $\tilde{\lambda}$ is strictly convex on the interior of the set where $\tilde{F}$ is finite.
\end{prop}

The proof is exactly the same.

\subsection{Asymptotic of characteristic functions}\label{SectionAsymptoticofcharacteristicfunctions}
We now introduce the characteristic function associated to our chain.
For $\xi \in \R^d$ and for $k,j\in \{1,...,N\}$ and $n\in \N$, one defines
$$\psi^{(n)}_{k,j}(\xi)=\sum_{x\in \Z^d}p^{(n)}_{k,j}(0,x)\mathrm{e}^{ix\cdot \xi}.$$
As announced, it will be easier in the following to deal with matrices, as in \cite{Woess1} (see paragraphs 6.A and 8.B).
Denote by $\psi(\xi)$ the matrix with entries $\psi^{(1)}_{k,j}(\xi)=:\psi_{k,j}(\xi)$.

In this section, we show a technical result which will lead to a central limit theorem in next section (see Theorem~\ref{TCL2}).
This technical result states convergence of the characteristic matrices $\psi^n(\frac{\xi}{\sqrt{n}})$ to Gaussian-like functions.
This is precisely Proposition~\ref{TCL}.
It will also allow us to show a local limit theorem (see Proposition~\ref{theorem1spitzer} below) which will be used to describe the Martin boundary.
The proofs of these will be based on Fourier analysis
and that is the reason for introducing the matrix $\psi(\xi)$ which is an analog of Fourier transforms.

\begin{lem}
The entries of the $n$th power $\psi^n(\xi)$ are $\psi^{(n)}_{k,j}(\xi)$.
\end{lem}

\begin{proof}
One has
\begin{align*}
    \psi^{(2)}_{k,j}(\xi)=\sum_{x\in \Z^d}p^{(2)}_{k,j}(0,x)\mathrm{e}^{ix\cdot \xi}&=\sum_{x\in \Z^d}\sum_{l\in \{1,...,N\}}\sum_{y\in \Z^d}p_{k,l}(0,y)p_{l,j}(y,x)\mathrm{e}^{i(x+y-y)\cdot \xi}\\
    &=\sum_{l\in \{1,...,N\}}\psi_{k,l}(\xi)\psi_{l,j}(\xi),
\end{align*}
which is the desired conclusion for $n=2$.
One concludes by induction.
\end{proof}

Let $\psi_{u}(\xi)$ be the characteristic matrix associated to the transition kernel $p_u$ and $\psi_{k,j;u}$ its entries.
By definition, 
$$\psi_{k,j;u}(\xi)=\sum_{x\in \Z^d}p_{k,j;u}(0,x)\mathrm{e}^{ix\cdot \xi}=\sum_{x\in \Z^d}p_{k,j}(0,x)\mathrm{e}^{u\cdot x}\mathrm{e}^{ix\cdot \xi}.$$

To prove our results below, it will be easier to center our chain.
Precisely, we define
$$\chi_{u}(\xi)=\psi_{u}(\xi)\mathrm{e}^{-i\nabla \lambda (u)\cdot\xi}.$$ 
This is the good choice of centering.
Indeed, interpreting the chain as a random walk in $\Z^d$ (that might jump from one level to another), we will see that this random walk goes in average in the direction of the spatial derivative of $\lambda$.
This is an analog of the law of large numbers (see Proposition~6.7 in \cite{Woess3}).

Let us now study the matrix $\chi_u$.
A direct calculation shows that $\chi_{u}(0)=F(u)$ and $\nabla_{\xi}\chi_{u}(0)=i\nabla F(u)-i\nabla \lambda (u)F(u)$.
The Hessian quadratic form at 0 is given by
$\nabla_{\xi}^2\chi_u(0)=-\nabla^2F(u)-(\nabla \lambda(u))^2F(u)+2\nabla \lambda (u)\nabla F(u)$, meaning that
\begin{align*}
[\nabla_{\xi}^2\chi_u(0)]_{k,j}(\theta)&=\sum_{x\in \Z^d}-(x\cdot \theta)^2p_{k,j;u}(0,x)-(\nabla \lambda(u)\cdot \theta)^2p_{k,j;u}(0,x)\\
&+2(x\cdot \theta)(\nabla \lambda (u)\cdot \theta)p_{k,j;u}(0,x)\\
&=-\sum_{x\in \Z^d}[(x-\nabla \lambda (u))\cdot \theta]^2p_{k,j}(0,x)\mathrm{e}^{u\cdot x}, \theta \in \R^d.
\end{align*}

The main result of this section is the following.
\begin{prop}\label{TCL}
With these notations, for every $\xi \in \R^d$ and for every $u\in H$, $[\chi_u(\frac{\xi}{\sqrt{n}})]^n$ converges to the matrix
$\mathrm{e}^{-\frac{1}{2}Q_u(\xi)}C(u)\cdot \nu(u)$
as $n$ tends to infinity, where $Q_u$ is a positive definite quadratic form.
Furthermore, the convergence is uniform in $\xi$ on compact sets of the form $\{\xi\in \R^d,\|\xi \|\leq A\}$.
It is also uniform in $u$ lying in a compact subset of $H$.
\end{prop}

To prove this proposition, we will first study the behavior of the dominant eigenvalue.
If $\xi\in \R^d$ is small enough, then the eigenvalue of $\chi_u(\xi)$ which has maximal absolute value is simple and every other eigenvalue has smaller absolute value.
Indeed, it is the case for $\xi=0$, since $\chi_u(0)=F(u)$.
Moreover, the eigenvalues of a matrix $A$ are continuous in $A$, so it is the case for small enough $\xi$.
Let $\lambda_u(\xi)$ be this dominant eigenvalue.
By definition, $\lambda_u(0)=\lambda(u)$.

As for $\lambda$, we use the fact that $\lambda_u(\xi)$ is simple to state that it is an analytic function in $\xi$, for small enough $\xi$.
Similarly, (right and left) eigenvectors associated to $\lambda_u(\xi)$ are analytic.
Let $C(\chi_u(\xi))=C_u(\xi)$ (on the right) and $\nu(\chi_u(\xi))=\nu_u(\xi)$ (on the left) be these eigenvectors.
We choose a normalization by declaring that
$\nu_u(0)\cdot C_u(\xi)=1$ and $\nu_u(\xi)\cdot C_u(\xi)=1$ so that $\lambda_u(\xi)=\nu_u(\xi)\chi_u(\xi)C_u(\xi)$.

\begin{lem}\label{lemmaTCL}
With these notations, if $u\in H$, $\lambda_u(\xi)=1-\frac{1}{2}Q_u(\xi)+o(\xi^2)$, where $Q_u$ is a positive definite quadratic form.
\end{lem}

\begin{proof}
This is just Taylor-Young theorem applied to $\lambda_u$ at $\xi=0$.
One has
$$\lambda_u=\lambda_u(0)+\nabla_{\xi} \lambda_u(0)\cdot \xi+\frac{1}{2}\nabla_{\xi}^2\lambda_u(0)(\xi)+o(\xi^2).$$
Since $u\in H$, $\lambda_u(0)=1$.
For the other terms, first write $\lambda_u(\xi)=\nu_u(\xi)\chi_u(\xi)C_u(\xi)$.
Then,
\begin{align*}\nabla_{\xi} \lambda_u(0)\cdot \xi=\nu_u(0)(\nabla_{\xi}\chi_u(0)\cdot \xi) C_u(0)&+(\nabla_{\xi}\nu_u(0)\cdot \xi)\chi_u(0)C_u(0)\\
&+\nu_u(0)\chi_u(0)(\nabla_{\xi}C_u(0)\cdot \xi).
\end{align*}
Replacing $\chi_u(0)$ with $F(u)$ and $\nabla_{\xi}\chi_u(0)$ with $i(\nabla F(u)-\nabla \lambda(u)F(u))$, as well as $\nu_u(0)$ and $C_u(0)$ with $\nu(u)$ and $C(u)$, one gets
\begin{align*}
    \nabla_{\xi} \lambda_u(0)&=i[\nu(u)\nabla F(u) C(u)-\nabla \lambda(u)\nu(u)F(u)C(u)]\\
    &+\lambda(u)[(\nabla_{\xi}\nu_u(0)C(u)+\nu(u)\nabla_{\xi}C_u(0)].
\end{align*}
Recall that $\nabla \lambda(u)=\nu(u)\nabla F(u)C(u)$ (see Equation~(\ref{equationlambdaprime})), so that the first term of
$\nabla_{\xi} \lambda_u(0)$ is zero.
Besides, differentiating $\nu_u(\xi)\cdot C_u(\xi)=1$ in $\xi$, one has
$$\nabla_{\xi}\nu_u(0)C(u)+\nu(u)\nabla_{\xi}C_u(0)=0,$$ so the second term is also zero.
Thus, $\nabla_{\xi} \lambda_u(0)=0$.

Now, for the second order term,
define $Q_u=-\nabla^2_{\xi}\lambda_u(0)$.
Fix $u\in H$ and
let $\tilde{F}(v)$ be the matrix with entries
$$\tilde{F}_{k,j}(v)=\sum_{x\in \Z^d}p_{k,j}(0,x)\mathrm{e}^{u\cdot x}\mathrm{e}^{v\cdot (x-\nabla \lambda(u))}.$$
The matrix $\tilde{F}(v)$ is strongly irreducible and one can apply Proposition~\ref{strictconvexity2}.
Denote by $\tilde{\lambda}(v)$ its dominant eigenvalue, then $\tilde{\lambda}(v)$ is strictly convex.
By definition, one has $\chi_u(\xi)=\tilde{F}(i\xi)$ so that if $\xi$ is small enough, then
$\lambda_u(\xi)=\tilde{\lambda}(i\xi)$.
Since $\lambda_u$ and $\tilde{\lambda}$ are analytic, one deduces that $\nabla^2_{\xi}\lambda_u(0)=-\nabla^2\tilde{\lambda}(0)$,
which ensures that $\nabla^2_{\xi}\lambda_u(0)$ is negative definite.
In other words, $Q_u$ is positive definite.
\end{proof}

We deduce from this lemma the following.
\begin{prop}\label{TCLvp}
With the same notations, for every $\xi \in \R^d$, for every $u\in H$, $[\lambda_u(\frac{\xi}{\sqrt{n}})]^n$ converges to $\mathrm{e}^{-\frac{1}{2}Q_u(\xi)}$.
The convergence is uniform in $\xi$ on compact sets of the form $\{\xi\in \R^d,\|\xi \|\leq A\}$.
It is also uniform in $u$ lying in a compact subset of $H$.
Furthermore, for large enough $n$, uniformly in $u$ lying in a compact subset $H_0$ of $H$ and in $\xi \in \{\xi\in \R^d,\|\xi\|\leq a\}$ (where $a$ is small enough and depends on $H_0$),  $|\lambda_u(\xi)|^n\leq \mathrm{e}^{-n\frac{1}{4}Q_u(\xi)}$.
\end{prop}

\begin{proof}
For every $u\in H$, $Q_u$ is a positive definite quadratic form.
In particular, there exists $\alpha_u>0$ such that for every $\xi\in \mathcal{Q}$, $Q_u(\xi)\geq \alpha_u\|\xi\|^2$.
Let $H_0$ be a compact subset of $H$.
For $u\in H_0$, one can choose $\alpha$ independent of $u$: for every $u\in H_0$, for every $\xi\in \mathcal{Q}$, $Q_u(\xi)\geq \alpha\|\xi\|^2$.
One deduces from Lemma~\ref{lemmaTCL} that
$$\frac{1-\lambda_u(\xi)}{Q_u(\xi)}\rightarrow 1/2,\xi \rightarrow 0.$$
This convergence is uniform in $u\in H_0$, for the lower bound $Q_u(\xi)\geq \alpha\|\xi\|^2$ is uniform in $u\in H_0$ and the term $o(\xi^2)$ in the Taylor series depends continuously on $u$.

Using real and imaginary parts $\Re$ and $\Im$, one gets
$$\frac{1-\Re(\lambda_u(\xi))}{Q_u(\xi)}\rightarrow 1/2,\xi \rightarrow 0, \text{ and }
\frac{\Im(\lambda_u(\xi))}{Q_u(\xi)}\rightarrow 0,\xi \rightarrow 0.$$
Since these limits are uniform in $u\in H_0$, one can choose small enough $a$, not depending on $u$, such that if $\|\xi\|\leq a$ and $u\in H_0$, then
$$\frac{1-\Re(\lambda_u(\xi))}{Q_u(\xi)}\geq\frac{1}{3}, \text{ and }
\frac{|\Im(\lambda_u(\xi))|}{Q_u(\xi)}\leq \frac{1}{12}.$$
One deduces that $0\leq \Re(\lambda(\xi))\leq 1-\frac{1}{3}Q_u(\xi)$ and that $|\Im(\lambda_u(\xi))|\leq \frac{1}{12}Q_u(\xi)$.
Combining those, one gets $|\lambda_u(\xi)|\leq 1-\frac{1}{4}Q_u(\xi)\leq \mathrm{e}^{-\frac{1}{4}Q_u(\xi)}$.
Thus, $|\lambda_u(\xi)|^n\leq~\mathrm{e}^{-n\frac{1}{4}Q_u(\xi)}$.

Finally, for large enough $n$, not depending on $u\in H_0$, and not depending on $\xi\in \{\xi\in \R^d,\|\xi \|\leq A\}$,
$\lambda_u(\frac{\xi}{\sqrt{n}})$ is well defined and stays in a ball centered at 1 with radius $r<1$.
In particular, one can choose uniformly in $u$ and in $\xi$ a complex logarithm determination and apply it to $\lambda_u(\frac{\xi}{\sqrt{n}})$.
Denote by $\mathrm{Log}$ such a logarithm.
Since $\lambda_u(\xi)=1-\frac{1}{2}Q_u(\xi)+o(\xi^2)$ with the $o(\xi^2)$ depending continuously on $u$, one has
$\lambda_u(\frac{\xi}{\sqrt{n}})=1-\frac{1}{2}\frac{1}{n}Q_u(\xi)+o(\frac{1}{n})$,
with the $o(\frac{1}{n})$ being uniform in $u$ and $\xi$.
In particular, $\mathrm{Log}([\lambda_u(\frac{\xi}{\sqrt{n}})]^n)$ uniformly converges to $-\frac{1}{2}Q_u(\xi)$ and so $[\lambda_u(\frac{\xi}{\sqrt{n}})]^n$ uniformly converges to
$\mathrm{e}^{-\frac{1}{2}Q_u(\xi)}$.
\end{proof}

We can now prove Proposition~\ref{TCL}
\begin{proof}
Fix a compact set $\{\xi\in \R^d,\|\xi \|\leq A\}$ and a compact subset $H_0$ of $H$.
For large enough $n$, independently of $u\in H_0$,
$\lambda_u(\frac{\xi}{\sqrt{n}})$ is well defined.
Denote by $\pi_u(\frac{\xi}{\sqrt{n}})$ the spectral projection on the associated eigenspace.
With our previous notations, $\pi_u(\frac{\xi}{\sqrt{n}})=C(\frac{\xi}{\sqrt{n}})\cdot \nu(\frac{\xi}{\sqrt{n}})$.
Use Theorem~\ref{theoremKato}
to write
$$\left [\chi_u\left (\frac{\xi}{\sqrt{n}}\right )\right ]^n=\left [\lambda_u\left (\frac{\xi}{\sqrt{n}}\right )\right ]^n\pi_u\left (\frac{\xi}{\sqrt{n}}\right )+\left [R_u\left (\frac{\xi}{\sqrt{n}}\right )\right ]^n,$$
where $R_u^n$ is a remainder whose norm is bounded, uniformly in $\xi$, by some $C\tilde{\lambda}_u^n$, with $0<\tilde{\lambda}_u<\lambda_u(0)$.
For $u\in H$, by definition, $\lambda_u(0)=1$.
Thus, since $H_0$ is compact, there exists $\tilde{\lambda}<1$, independent of $u\in H_0$, such that the norm of $R_u^n$ is bounded by $C\tilde{\lambda}^n$.
The projection $\pi$ is continuous with respect to $u$ and with respect to $\xi$, so that $\pi_u(\frac{\xi}{\sqrt{n}})$ uniformly converges to $\pi_u(0)=C(u)\cdot \nu(u)$.
Using Proposition~\ref{TCLvp}, one gets that $\left [\chi_u\left (\frac{\xi}{\sqrt{n}}\right )\right ]^n$ converges to $\mathrm{e}^{-\frac{1}{2}Q_u(\xi)}C(u)\cdot \nu(u)$.
\end{proof}

Proposition~\ref{TCL} is the result we will really use in the following.
We will also deduce from it a central limit theorem in Section~\ref{SectionMarkovchainsonathickenedlattice}.

\subsection{Martin compactification}\label{SectionMartincompactification}
The goal of this section is to fully describe the Martin boundary of our transition kernel on $\Z^d\times \{1,...,N\}$.
We will adapt the proof of P.~Ney and F.~Spitzer, given in \cite{NeySpitzer}.
There will be several steps.
First, we find an asymptotic of $p_{k,j;u}^{(n)}(0,x)$ when $n$ goes to infinity.
We then deduce an asymptotic of the Green function $G_{k,j;u}(x,y)$ when $y$ tends to infinity.
Finally, we show that $G((x,k),(y,j))/G((x_0,k_0),(y,j))$ converges when $y$ tends to infinity and converges in direction, i.e. $\frac{y}{\|y\|}$ converges to some point on the sphere.
Besides, the limit is independent of $j$ and is continuous with respect to every variable.

Rather than working with each entry $p_{k,j;u}^{(n)}(0,x)$, we will work with matrices, as is suggested by Proposition~\ref{TCL}.
Since the quadratic form $Q_u$ is positive definite, the corresponding symmetric matrix $\tilde{Q}_u$ is invertible.
Let $\tilde{\Sigma}_u$ be its inverse. It is again a symmetric matrix and we can define the associated quadratic form $\Sigma_u$.
Denote by $|Q_u|$ and $|\Sigma_u|$ the corresponding determinants.
Recall that $P_u(0,x)$ is the matrix with entries $p_{k,j}(0,x)\mathrm{e}^{u\cdot x}$ and define $P^n_u(0,x)$ as the matrix with entries $p^{(n)}_{k,j}(0,x)\mathrm{e}^{u\cdot x}$.
Recall that $\mathcal{F}$ is the interior of the set of $u\in \R^d$ such that every entry of the matrix $F(u)$ is finite and recall that
$$H=\{u\in \mathcal{F},\lambda(u)=1\}.$$

\begin{prop}\label{theorem1spitzer}
Let $p$ be a strongly irreducible $\Z^d$-invariant transition kernel on $\Z^d\times \{1,...,N\}$.
With the same notations as before, one defines, for $u\in H$ and $x\in \Z^d$,
$$A_n(x,u)=(2\pi n)^{\frac{d}{2}}P_u^{(n)}(0,x)-|Q_u|^{-\frac{1}{2}}\mathrm{e}^{-\frac{1}{2n}\Sigma_u(x-n\nabla \lambda (u))}C(u)\cdot \nu(u).$$
Then, for every $x\in \Z^d$ and every $u\in H$, $A_n(x,u)$ converges to 0 when $n$ tends to infinity.
Furthermore, the convergence is uniform in $x\in \Z^d$ and in $u$ lying in a compact subset of $H$.
\end{prop}

\begin{proof}
Fix a compact subset $H_0$ of $H$.
Since $\psi_{k,j;u}$ is the characteristic function associated to $p_{k,j;u}$,
the Fourier inversion formula gives
$$p_{k,j;u}^{(n)}(0,x)=\frac{1}{(2\pi)^d}\int_{\mathbf{C}}\psi_{k,j;u}^{(n)}(\xi)\mathrm{e}^{-i\xi\cdot x}\mathrm{d}\xi,$$
where $\mathbf{C}$ is the unit cube $\{\xi\in \R^d,\xi=(\xi_1,...,\xi_d),|\xi_j|\leq \pi,j=1,...,d\}$.
In the following, if $T$ is a matrix whose entries are functions $t_{k,j}$, we will denote by
$\displaystyle \int T$ the matrix whose entries are $\displaystyle \int t_{k,j}$.
With this notation, one thus has
$$P_u^{(n)}(0,x)=\frac{1}{(2\pi)^d}\int_{\mathbf{C}}\psi_u^{(n)}(\xi)\mathrm{e}^{-i\xi \cdot x}\mathrm{d}\xi=\frac{1}{(2\pi)^d}\frac{1}{n^{\frac{d}{2}}}\int_{\sqrt{n}\mathbf{C}} \psi_u^{(n)}\left (\frac{\xi}{\sqrt{n}}\right )\mathrm{e}^{-i\frac{\xi\cdot x}{\sqrt{n}}}\mathrm{d}\xi.$$
Replacing $\psi_u$ with the centered characteristic matrix $\chi_u$, one gets
$$(2\pi n)^{\frac{d}{2}}P_u^{(n)}(0,x)=\frac{1}{(2\pi)^{\frac{d}{2}}}\int_{\sqrt{n}\mathbf{C}} \left [\chi_u\left (\frac{\xi}{\sqrt{n}}\right )\right ]^n\mathrm{e}^{-i\frac{\xi\cdot (x-n\nabla\lambda(u))}{\sqrt{n}}}\mathrm{d}\xi.$$

As in \cite{NeySpitzer}, we cut this integral in five parts.
Let $A\geq 0$ be some non-negative real number and $0<\alpha<1$ another real number.
We write
$$(2\pi n)^{\frac{d}{2}}P_u^{(n)}(0,x)=I_0(n)+I_1(n,A)+I_2(n,A)+I_3(n,A,\alpha)+I_4(n,\alpha),$$
with
$$I_0(n)=\frac{1}{(2\pi)^{\frac{d}{2}}}\int_{\R^d}\mathrm{e}^{-\frac{1}{2}Q_u(\xi)}C(u)\cdot \nu(u)\mathrm{e}^{-i\frac{\xi\cdot (x-n\nabla\lambda(u))}{\sqrt{n}}}\mathrm{d}\xi,$$
$$I_1(n,A)=\frac{1}{(2\pi)^{\frac{d}{2}}}\int_{\|\xi\|\leq A}\left ( \left [\chi_u\left (\frac{\xi}{\sqrt{n}}\right )\right ]^n-\mathrm{e}^{-\frac{1}{2}Q_u(\xi)}C(u)\cdot \nu(u) \right )\mathrm{e}^{-i\frac{\xi\cdot (x-n\nabla\lambda(u))}{\sqrt{n}}}\mathrm{d}\xi,$$
$$I_2(n,A)=\frac{-1}{(2\pi)^{\frac{d}{2}}}\int_{\|\xi\|>A}\mathrm{e}^{-\frac{1}{2}Q_u(\xi)}C(u)\cdot \nu(u)\mathrm{e}^{-i\frac{\xi\cdot (x-n\nabla\lambda(u))}{\sqrt{n}}}\mathrm{d}\xi,$$
$$I_3(n,A,\alpha)=\frac{1}{(2\pi)^{\frac{d}{2}}}\int_{A< \|\xi\|\leq \sqrt{n}\alpha} \left [\chi_u\left (\frac{\xi}{\sqrt{n}}\right )\right ]^n\mathrm{e}^{-i\frac{\xi\cdot (x-n\nabla\lambda(u))}{\sqrt{n}}}\mathrm{d}\xi,$$
$$I_4(n,\alpha)=\frac{1}{(2\pi)^{\frac{d}{2}}}\int_{\sqrt{n}\alpha<\|\xi\|,\xi\in \sqrt{n}\mathbf{C}} \left [\chi_u\left (\frac{\xi}{\sqrt{n}}\right )\right ]^n\mathrm{e}^{-i\frac{\xi\cdot (x-n\nabla\lambda(u))}{\sqrt{n}}}\mathrm{d}\xi.$$
A direct calculation shows that
$$I_0(n)=|Q_u|^{-\frac{1}{2}}\mathrm{e}^{-\frac{1}{2n}\Sigma_u(x-n\nabla \lambda (u))}C(u)\cdot \nu(u).$$
What is left to do is showing that the integrals $I_1,I_2,I_3$ and $I_4$ can be bounded by some arbitrary $\epsilon>0$, uniformly in $x\in \Z^d$ and in $u\in H_0$.
Actually, we will prove the same thing but replacing integrands with their absolute value, so that uniformity with respect to $x$ is obvious.
Precisely, for fixed $\epsilon>0$, we will choose numbers $A$ and $\alpha$
so that each integral is smaller than $\epsilon$ (with absolute values) and for large enough $n$, say $n\geq n_0$.
It will then suffice to choose these numbers $A$, $\alpha$ and $n_0$ independently of $u\in H_0$ to conclude.

We will work on integrals $I_1,I_2,I_3$ and $I_4$ in this order, but the choice of $A$, $\alpha$ and $n_0$ will finally be as follows.
One chooses $\alpha$ small enough and $A$ large enough to bound $|I_3|$, then $A$ large enough to bound $|I_2|$, then $n_0$ large enough to bound $|I_1|$ and  $|I_4|$.

We first deal with the integral $I_1(n,A)$.
According to Proposition~\ref{TCL}, we have that $[\chi_u(\frac{\xi}{\sqrt{n}})]^n-\mathrm{e}^{-\frac{1}{2}Q_u(\xi)}C(u)\cdot \nu(u)$ converges to 0 when $n$ tends to infinity, uniformly in $u\in H_0$ and uniformly on $\{\xi\in \R^d,\|\xi \|\leq A\}$.
Thus, $|I_1(n)|\leq \epsilon$ for $n\geq n_1$, with $n_1$ not depending on $u\in H_0$.

To study $I_2(n,A)$, recall that since $Q_u$ is positive definite, $Q_u(\xi)\geq a \|\xi\|^2$, with $a$ independent of $u\in H_0$.
We get an upper bound for $I_2(n,A)$:
$$|I_2(n,A)|\leq K_0\int_{\|\xi\|>A}\mathrm{e}^{-\frac{1}{2}a\|\xi\|^2}\mathrm{d}\xi$$
with $K_0$ a real number.
The integrand is then an integrable and continuous function, so that for large enough $A$, $|I_2(n,A)|$ is bounded by $\epsilon$.

In $I_3(n,A,\alpha)$,
we can choose $\alpha$ small enough so that the dominant eigenvalue $\lambda_u(\frac{\xi}{\sqrt{n}})$ is well defined for every $u\in H_0$ and for every $\xi$ such that $\|\xi\|\leq \sqrt{n}\alpha$.
Let us use again Theorem~\ref{theoremKato}
to write
$$\left [\chi_u\left (\frac{\xi}{\sqrt{n}}\right )\right ]^n=\left [\lambda_u\left (\frac{\xi}{\sqrt{n}}\right )\right ]^n\pi_u\left (\frac{\xi}{\sqrt{n}}\right )+\left [R_u\left (\frac{\xi}{\sqrt{n}}\right )\right ]^n$$
where $R_u^n$ is a remainder whose norm is bounded, uniformly in $\xi$, by some $C\tilde{\lambda}_u^n$, with $0<\tilde{\lambda}_u<\lambda_u(0)=1$.
Since $H_0$ is compact, there exists $\tilde{\lambda}<1$ independent of $u$ such that the norm of $R_u^n$ is bounded by $C\tilde{\lambda}^n$.
Thus, we have
\begin{equation}\label{equationI3}
\begin{split}
    |I_3(n,A,\alpha)|&\leq \int_{A<\|\xi\|\leq \sqrt{n}\alpha}\left [\lambda_u\left (\frac{\xi}{\sqrt{n}}\right )\right ]^n\left \|\pi_u\left (\frac{\xi}{\sqrt{n}}\right )\right \|\mathrm{d}\xi\\
    &+\int_{A<\|\xi\|\leq \sqrt{n}\alpha}\left \|\left [R_u\left (\frac{\xi}{\sqrt{n}}\right )\right ]^n\right \|\mathrm{d}\xi.
\end{split}
\end{equation}
Using Proposition~\ref{TCLvp}, we bound the first integral by
$$K_1\int_{A<\|\xi\|}\mathrm{e}^{-\frac{1}{4}Q_u(\xi)}\mathrm{d}\xi,$$
where $K_1$ is a real number.
Again, since $Q_u(\xi)\geq a \|\xi\|^2$,
we can bound this integral by
$$K_1\int_{A<\|\xi\|}\mathrm{e}^{-\frac{1}{4}a \|\xi\|^2}\mathrm{d}\xi.$$
We then bound the second integral in (\ref{equationI3}) by the volume of the ball of radius $\sqrt{n}\alpha$ multiplied by $\tilde{\lambda}^n$, i.e. we bound it by $K_2\sqrt{n}^d\tilde{\lambda}^n$.
We can choose $n_2$ and $A$ independently of $u\in H_0$ such that $|I_3(n,A,\alpha)|\leq \epsilon$ for $n\geq n_2$.

Finally, in the last integral, we cannot speak about a dominant eigenvalue,
but we can still speak about an eigenvalue that has maximal absolute value.
Let $\lambda$ be such an eigenvalue.
We first write
$$I_4(n,\alpha)=n^{\frac{d}{2}}\frac{1}{(2\pi)^{\frac{d}{2}}}\int_{\alpha<\|\xi\|,\xi\in \mathbf{C}} [\chi_u(\xi)]^n\mathrm{e}^{-i\xi\cdot (x-n\nabla\lambda(u))}\mathrm{d}\xi.$$

We prove by contradiction that $|\lambda|< 1$, for $\xi \neq0, \xi \in \mathbf{C}$.
Assume on the contrary that $|\lambda|\geq 1$.
Let $v$ be a left eigenvector associated to $\lambda$ and $|v|$ be the vector whose coordinates are the absolute values of those of $v$.
We show that $|v|$ is a left eigenvector for $\chi_u(0)=F(u)$.
Indeed, one has $|v|\leq |\lambda||v|\leq |v|F(u)$, i.e. these inequalities are true for every coordinate.
Recall that $C(u)$ is a right eigenvector for $F(u)$, associated to $\lambda(u)$.
Every coordinate of $C(u)$ is positive.
We use the norm on $\R^d$ defined by the formula $\|w\|=|w|\cdot C(u)$.
If $u\in H$, that is if $\lambda(u)=1$, one always has $\|wF(u)\|\leq\|w\|$.
Since $|v|\leq |v|F(u)$, one necessarily has $|v|F(u)=|v|$.
In particular, every coordinate of $|v|$ is positive and since $|v|\leq |\lambda ||v|\leq |v|F(u)=|v|$, one necessarily has $|\lambda|=1$.

We now use strong irreducibility.
Fix $\xi\neq 0$ and $\xi\in \mathbf{C}$.
There exists at least one vector $e_j$ in the canonical basis of $\R^d$ such that $e_j\cdot \xi \neq 0$.
Without loss of generality, we can assume that $e_1\cdot \xi \neq 0$.
Fix two indices $k$ and $j$ and an integer $n$ such that $p_{k,j}^{(n)}(0,0)\neq 0$, $p_{k,j}^{(n)}(0,e_1)\neq 0$.
Write $v\chi_u(\xi)^n=\lambda^n v$ and $|v|\chi_u(0)^n=|v|$.
One has
$$|v_j|=\left |\sum_kv_k[\chi_u(\xi)]^n_{k,j}\right |=\left |\sum_k\sum_{x\in \Z^d}v_kp_{k,j;u}^{(n)}(0,x)\mathrm{e}^{ix\cdot \xi}\right |,$$
so
$$|v_j|\leq \sum_k|v_k|\left |\sum_{x\in \Z^d}p_{k,j;u}^{(n)}(0,x)\mathrm{e}^{ix\cdot\xi}\right |.$$
Since the arguments of $\mathrm{e}^{0}$ and $\mathrm{e}^{ie_1\cdot\xi}$ are different,
$$\left |\sum_{x\in \Z^d}p_{k,j;u}^{(n)}(0,x)\mathrm{e}^{ix\cdot\xi}\right |<\sum_{x\in \Z^d}p_{k,j;u}^{(n)}(0,x).$$
Since $|v_k|\neq 0$, one deduces that
$$|v_j|<\sum_k|v_k|\sum_xp_{k,j;u}^{(n)}(0,x)=|v_j|.$$
This is a contradiction, so that $|\lambda|<1$.

Thus, the spectral radius of $\chi_u(\xi)$ is smaller than 1, for $\xi\in \mathbf{C}$, $\xi\neq0$.
Using compactness, for $u\in H_0$ and $\|\xi\|\geq \alpha$, one can uniformly bound the norm of $[\chi_u(\xi)]^n$ by some $C\delta^n$, with $\delta<1$.
Thus, $|I_4(n,\alpha)|\leq (2\pi n)^{\frac{d}{2}}\delta^n$, so $|I_4(n,\alpha)|\leq \epsilon$ for $n\geq n_3$, with $n_3$ independent of $u$.
\end{proof}

\begin{rem}\label{remarkassumptions}
We really used strong irreducibility when bounding the integral $I_4$.
\end{rem}

The following proposition is a slight refinement of Proposition~\ref{theorem1spitzer}.
\begin{prop}\label{theorem1'spitzer}
Let $p$ be a strongly irreducible $\Z^d$-invariant transition kernel on $\Z^d\times \{1,...,N\}$.
Let $\gamma_0\geq 0$.
If $\gamma$ is a real number between 0 and $\gamma_0$ and if $u\in H$ and $x\in \Z^d$, one defines
$$\tilde{A}_n(x,u,\gamma)=\left ( \frac{\|x-n\nabla \lambda (u)\|}{\sqrt{n}} \right )^{\gamma}A_n(x,u).$$
Then, for every $x\in \Z^d$, for every $u\in H$ and for every $\gamma \in [0,\gamma_0]$, $\tilde{A}_n(x,u,\gamma)$ converges to 0 when $n$ tends to infinity.
Furthermore, the convergence is uniform in $x\in \Z$, in $u$ lying in a compact subset of $H$ and in $\gamma\in [0,\gamma_0]$.
\end{prop}

\begin{rem}
This is a generalization of \cite[Theorem~2.2]{NeySpitzer}.
In this article, the authors only deal with $\gamma \in [0,2d]$.
Actually, they do not need to use the fact that $\gamma\leq 2d$ in their proof.
However, we will only use the result for $\gamma \in [0,2d]$ in the following.
\end{rem}

\begin{proof}
Fix a compact subset $H_0$ of $H$.
Notice that if $0\leq \gamma \leq \gamma_0$, then
$$\mathrm{min}[\tilde{A}_n(x,u,0),\tilde{A}_n(x,u,\gamma_0)]\leq \tilde{A}_n(x,u,\gamma)\leq \mathrm{max}[\tilde{A}_n(x,u,0),\tilde{A}_n(x,u,\gamma_0)].$$
It thus suffices to show that $\tilde{A}_n(x,u,0)$ and $\tilde{A}_n(x,u,\gamma_0)$ both converge to 0, uniformly in $x$ and $u\in H_0$.
Since $\tilde{A}_n(x,u,0)=A_n(x,u)$, one only has to deal with $\tilde{A}_n(x,u,\gamma_0)$.
In other words, to get uniform convergence in $\gamma$, we only have to prove convergence for fixed $\gamma$.
Moreover, the inequality above shows that if one gets convergence of $\gamma_0$, then one gets convergence for every $\gamma \leq \gamma_0$.
It will be more convenient to deal with $\gamma \in 2\N$ in the following, which is sufficient to conclude.
We thus assume that $\gamma=2k,k\in \N$.
Denote by $\Delta_{\xi}$ the Laplace operator $\sum\frac{\partial^2}{\partial\xi_j^2}$ and $\Delta_{\xi}^k$ the Laplace operator iterated $k$ times.

Again, we use an integral formula.
We have to study
$$(2\pi n)^{\frac{d}{2}}\left ( \frac{\|x-n\nabla \lambda (u)\|}{\sqrt{n}} \right )^{\gamma}P_u^{(n)}(0,x).$$
We write
\begin{align*}
    (2\pi n)^{\frac{d}{2}}&\left ( \frac{\|x-n\nabla \lambda (u)\|}{\sqrt{n}} \right )^{\gamma}P_u^{(n)}(0,x)\\
    &=\frac{1}{(2\pi)^{\frac{d}{2}}}\left ( \frac{\|x-n\nabla \lambda (u)\|}{\sqrt{n}} \right )^{\gamma}\int_{\sqrt{n}\mathbf{C}} \left [\chi_u\left (\frac{\xi}{\sqrt{n}}\right )\right ]^n\mathrm{e}^{-i\frac{\xi\cdot (x-n\nabla\lambda(u))}{\sqrt{n}}}\mathrm{d}\xi.
\end{align*}
Since $\gamma=2k$, an integration by parts shows that
\begin{align*}
    (2\pi n)^{\frac{d}{2}}&\left ( \frac{\|x-n\nabla \lambda (u)\|}{\sqrt{n}} \right )^{\gamma}P_u^{(n)}(0,x)\\
    &=\frac{1}{(2\pi)^{\frac{d}{2}}}\int_{\sqrt{n}\mathbf{C}} \Delta_{\xi}^k \left (\left [\chi_u\left (\frac{\xi}{\sqrt{n}}\right )\right ]^n\right )\mathrm{e}^{-i\frac{\xi\cdot (x-n\nabla\lambda(u))}{\sqrt{n}}}\mathrm{d}\xi.
\end{align*}

We then have to show some avatar of Proposition~\ref{theorem1spitzer}, but replacing $I_0$, $I_1$, $I_2$, $I_3$ and $I_4$ with integrals $J_0$, $J_1$, $J_2$, $J_3$ and $J_4$ in which we replace $[\chi_u(\frac{\xi}{\sqrt{n}})]^n$ with $\Delta_{\xi}^k([\chi_u(\frac{\xi}{\sqrt{n}})]^n)$ and $\mathrm{e}^{-\frac{1}{2}Q_u(\xi)}$ with $\Delta_{\xi}^k(\mathrm{e}^{-\frac{1}{2}Q_u(\xi)})$.
We need the same estimates on $\Delta_{\xi}^k([\chi_u(\frac{\xi}{\sqrt{n}})]^n)$ than the ones we used on $[\chi_u(\frac{\xi}{\sqrt{n}})]^n$, proving Proposition~\ref{theorem1spitzer}.
For simplicity, we deal with $k=1$, the general case is similar.

Let us first show that $\Delta_{\xi} ([\lambda_u(\frac{\xi}{\sqrt{n}})]^n)$ uniformly converges to $\Delta_{\xi} (\mathrm{e}^{-\frac{1}{2}Q_u(\xi)})$.
Indeed, differentiating twice, we get
\begin{equation}\label{equationlaplaceeigenvalue}
    \Delta_{\xi}([\lambda_u(\frac{\xi}{\sqrt{n}})]^n)=(n-1)\lambda_u(\frac{\xi}{\sqrt{n}})^{n-2}\|\nabla_{\xi} \lambda_u(\frac{\xi}{\sqrt{n}})\|^2+\lambda_u(\frac{\xi}{\sqrt{n}})^{n-1}\Delta_{\xi} \lambda_u(\frac{\xi}{\sqrt{n}}).
\end{equation}
However,
$$\lambda_u(\frac{\xi}{\sqrt{n}})=1-\frac{1}{2}Q_u(\frac{\xi}{\sqrt{n}})+o(\frac{\xi^2}{n}),$$
so that
$n\|\nabla_{\xi} \lambda_u(\frac{\xi}{\sqrt{n}})\|^2$ converges to $\frac{1}{4}\|\nabla_{\xi} Q_u(\xi)\|^2$ when $n$ tends to infinity, uniformly in $u\in H_0$ and in $\xi$ lying in a compact set.
Similarly, $\Delta_{\xi} \lambda_u(\frac{\xi}{\sqrt{n}})$ uniformly converges to $\Delta_{\xi} Q_u(\xi)$, which allows us to conclude.

Besides, Equation~\ref{equationlaplaceeigenvalue}
and Proposition~\ref{TCLvp} together show that
$$|\Delta_{\xi}([\lambda_u(\frac{\xi}{\sqrt{n}})]^n)|\leq K_0\mathrm{e}^{-\frac{1}{4}Q_u(\xi)},$$
for $\xi$ lying in a compact set and for large enough $n$, uniformly in $u\in H_0$.
We then use Theorem~\ref{theoremKato} to write
$$\left [\chi_u\left (\frac{\xi}{\sqrt{n}}\right )\right ]^n=\left [\lambda_u\left (\frac{\xi}{\sqrt{n}}\right)\right ]^n\pi_u\left (\frac{\xi}{\sqrt{n}}\right )+\left [R_u\left (\frac{\xi}{\sqrt{n}}\right )\right ]^n.$$
Differentiating twice, we get
\begin{align*}
    \Delta_{\xi}\left (\left [\chi_u\left (\frac{\xi}{\sqrt{n}}\right )\right ]^n\right )={}&\Delta_{\xi}\left (\left [\lambda_u\left (\frac{\xi}{\sqrt{n}}\right )\right ]^n\right )\pi_u\left (\frac{\xi}{\sqrt{n}}\right )\\
    &+\nabla_{\xi}\left (\left [\lambda_u\left (\frac{\xi}{\sqrt{n}}\right )\right ]^n\right )\cdot \nabla_{\xi}\left (\pi_u\left (\frac{\xi}{\sqrt{n}}\right )\right )\\
    &+ \left [\lambda_u\left (\frac{\xi}{\sqrt{n}}\right )\right ]^n\Delta_{\xi}\left (\pi_u\left (\frac{\xi}{\sqrt{n}}\right )\right )+ \Delta_{\xi} \left (\left [R_u\left (\frac{\xi}{\sqrt{n}}\right )\right ]^n\right ).
\end{align*}
Moreover,
$$\nabla_{\xi}([\lambda_u(\frac{\xi}{\sqrt{n}})]^n)=n[\lambda_u(\frac{\xi}{\sqrt{n}})]^{n-1}\frac{1}{\sqrt{n}}\nabla_{\xi} \lambda_u(\frac{\xi}{\sqrt{n}})$$
and
$$\nabla_{\xi} (\pi_u(\frac{\xi}{\sqrt{n}}))=\frac{1}{\sqrt{n}}\nabla_{\xi} \pi_u(\frac{\xi}{\sqrt{n}}),$$
so that the multiplication of these terms gives
$[\lambda_u(\frac{\xi}{\sqrt{n}})]^{n-1} \nabla_{\xi} \lambda_u(\frac{\xi}{\sqrt{n}})\cdot \nabla_{\xi} \pi_u(\frac{\xi}{\sqrt{n}})$.
In particular, it converges to $\mathrm{e}^{-\frac{1}{2}Q_u(\xi)}\nabla_{\xi} \lambda_u(0)\cdot \nabla_{\xi} \pi_u(0)$.
However, $\nabla_{\xi} \lambda_u(0)=0$, so it converges to 0.
We also have $\Delta_{\xi}(\pi_u(\frac{\xi}{\sqrt{n}}))=\frac{1}{n}\Delta_{\xi}\pi_u(\frac{\xi}{\sqrt{n}})$.
Moreover, the Laplace operator applied to $R_u^n$ still converges to 0 exponentially fast.
Thus, $\Delta_{\xi}([\chi_u(\frac{\xi}{\sqrt{n}})]^n)$ uniformly converges to $\Delta_{\xi} (\mathrm{e}^{-\frac{1}{2}Q_u(\xi)})\pi_u(0)$.

We can then bound $J_1$ in the same manner as we bounded $I_1$.

Integral $J_2$ can be bounded exactly in the same manner as $I_2$. Since $Q_u$ is positive definite, $\Delta_{\xi} (\mathrm{e}^{-\frac{1}{2}Q_u(\xi)})$ is integrable, uniformly in $u\in H_0$.

Moreover, since $|\Delta_{\xi}([\lambda_u(\frac{\xi}{\sqrt{n}})]^n)|\leq K_0\mathrm{e}^{-\frac{1}{4}Q_u(\xi)}$ and since $ \Delta_{\xi} ([R_u(\frac{\xi}{\sqrt{n}})]^n)$ converges to 0 exponentially fast, we can bound $J_3$ as we bounded $I_3$.

Finally, for $J_4$, notice that $\Delta_{\xi}([\chi_u(\theta)]^n)$ is bounded in norm by $K_1\|[\chi_u(\theta)]^n\|$
which is exponentially small if $\theta$ is bounded away from 0.
Thus, we can bound $J_4$ as we bounded $I_4$.

We deal similarly with the $k$ times iterated Laplace operator to conclude.
\end{proof}

We now make two assumptions about our chain.
Recall that we assume that the chain is strongly irreducible and we denote
$$H=\{u\in \mathcal{F},\lambda(u)=1\}.$$
Define also
$$D=\{u\in \mathcal{F},\lambda(u)\leq 1\}.$$

\begin{hyp}\label{hyp1}
The set $D$ is compact.
\end{hyp}

\begin{lem}\label{lemmafinitesupport}
If $p$ has finite support, then Assumption~\ref{hyp1} is satisfied.
\end{lem}

\begin{proof}
If $p$ has finite support, then $F$ is defined everywhere, that is $\mathcal{F}=\R^d$.
To show that $D$ is compact, we only have to show that $\lambda(u)$ tends to infinity when $u$ tends to infinity.
Now, if $\|u\|$ tends to infinity, one can assume that $u$ converges in direction, meaning that $\frac{u}{\|u\|}$ converges to some point $\theta$ on the unit sphere.
Thus, one can find $x\in \Z^d$ such that $u\cdot x$ tends to infinity (It suffices to choose $x\cdot \theta>0$).
Using strong irreducibility, one can choose some integer $n$ such that for every $k,j$, $p^{(n)}_{k,j}(0,x)> \delta>0$.
Then, every entry of $F(u)^n$ is larger than $\delta \mathrm{e}^{u\cdot x}$.
Since $\lambda(u)^n=\nu(u)F(u)^nC(u)$, $\lambda(u)^n> \delta \mathrm{e}^{u\cdot x} \nu(u)C(u)=\delta \mathrm{e}^{u\cdot x}$,
so that $\lambda(u)^n$ tends to infinity and so does $\lambda(u)$.
\end{proof}

Since the chains we study in this article have a finite support, Assumption~\ref{hyp1} will always be satisfied according to Lemma~\ref{lemmafinitesupport}.
However, as in \cite{NeySpitzer}, one does not need that the support is finite (see Condition~1.4 in \cite{NeySpitzer}).
This could be useful in another context.

\begin{rem}
Notice that since $D$ is compact, $H$ is compact.
We stated several results of convergence for functions depending on a parameter $u\in H$ and stated that the convergence was uniform in $u$ lying in a compact subset of $H$.
Uniformity is now true for $u\in H$.
\end{rem}

The second assumption is the following.
\begin{hyp}\label{hyp2}
The minimum of the function $\lambda$ is strictly smaller than 1.
\end{hyp}

\begin{rem}
Assumptions~\ref{hyp1} and~\ref{hyp2} only make sense if $\lambda(u)$ is well defined, which is the case if the chain is strongly irreducible.
\end{rem}

\begin{rem}\label{remarktransience}
As we will see in the following (see Proposition~\ref{corotransience} below), this assumption ensures that the Green function is finite.
\end{rem}

Recall that $H$ is the set of $u\in \R^d$ such that $\lambda(u)=1$
and $D$ is the set of $u\in \R^d$ such that $\lambda(u)\leq 1$.
Since 1 is not the minimum of $\lambda$ and since this function is strictly convex, the gradient of $\lambda$, which is denoted by $\nabla \lambda$ is non-zero on $H$.
Moreover, under these assumptions, $H$ is non-empty and is homeomorphic to the unit sphere $\Ss^{d-1}$.

\begin{lem}\label{explicithomeo}
An explicit homeomorphism is given by
$$u\in H\mapsto \frac{\nabla \lambda (u)}{\|\nabla \lambda (u)\|}\in \Ss^{d-1}.$$
\end{lem}

\begin{proof}
It is a convex analysis argument.
It can be deduced from the fact that $\nabla \lambda$ is non-zero on $H$ and that $\lambda$ is strictly convex and defined on a convex neighborhood of $D$.
For more details, we refer to \cite[Proposition~II.4.4]{Hennequin}.
\end{proof}

\begin{rem}
Assumptions~\ref{hyp1} and~\ref{hyp2} are a bit technical, since their statement involve the set $\mathcal{F}$, which is the interior of the set of $u\in \R^d$ such that every entry of $F$ is finite.
Actually, the assumption we really want to hold is the conclusion of Lemma~\ref{explicithomeo}.
Assumption~\ref{hyp1} should be compared with Condition~1.4 in \cite{NeySpitzer} which states that every point $u$ of $H$ has a neighborhood in $\R^d$ in which the matrix $F(u)$ is finite.
Both of these assumptions are generalizations of the condition that the chain is finitely supported.

In \cite{NeySpitzer}, the authors deduce from Condition~1.4 the conclusion of Lemma~\ref{explicithomeo}, although their conclusion is wrong
and one really needs the formulation of Assumption~\ref{hyp1}.
Indeed consider the strictly convex function $f:x\mapsto (x+\frac{1}{\sqrt{2}})^2+\frac{1}{2}$ defined on $\{x\in \R,x\geq -1\}$ whose graph is given by the following picture.

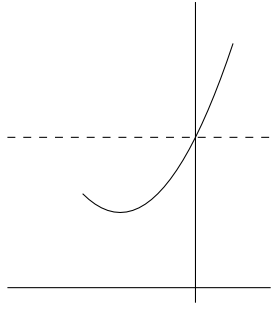
\begin{center}
\begin{tikzpicture}
\draw plot[domain=0:2] (\x,\x^2+1/4-\x);
\draw (-1,-1)--(2.5,-1);
\draw (3/2,-1.2)--(3/2,2.8);
\draw[dashed] (-1,1)--(2.5,1);
\end{tikzpicture}
\end{center}

The sub-level set $\{x,f(x)\leq 1\}$ is compact, as well as the level set $\{x,f(x)=1\}$, but the interior of the set on which $f$ is well defined is the set $\{x>-1\}$.
The analog of the set $D$, which is $\{x>-1,f(x)\leq 1\}$ is not compact.
Here, the analog of the set $H$, $\{x>-1,f(x)=1\}$ is not homeomorphic to the sphere $\Ss^0\simeq\{-1,1\}$ since it contains only one point.
This example seems artificial, since $f$ could be extended to $\R$, but actually, depending on the chain $p$, the set of $u\in \R^d$ such that the matrix $F(u)$ has finite entries could be contained in a proper closed subset of $\R^d$ and such a situation could arise.
\end{rem}

\begin{prop}\label{corotransience}
Under Assumptions~\ref{hyp1} and \ref{hyp2}, the Green functions is finite.
\end{prop}

\begin{proof}
Let us now use Assumption~\ref{hyp2}.
Fix $u\in H$, which is non-empty by Lemma~\ref{explicithomeo}. Then $\nabla \lambda(u)$ is non-zero.
For $\gamma=2d$, $\tilde{A}_n(x,u,\gamma)$ converges to 0, uniformly in $x$.
In particular, for large enough $n$,
$$\left \|(2\pi n)^{\frac{d}{2}}P_u^{(n)}(0,x)-|Q_u|^{-\frac{1}{2}}\mathrm{e}^{-\frac{1}{2n}\Sigma_u(x-n\nabla \lambda (u))}C(u)\cdot \nu(u)\right \|\leq \frac{C_1}{n^{d}},$$
where $C_1$ is some real number.
Now fix $x\in \Z^d$.
For large enough $n$, we have
$$\left | p_{k,j;u}^{(n)}(0,x)-\frac{C_2}{n^{\frac{d}{2}}}\mathrm{e}^{-C_3n}\right |\leq \frac{C'_1}{n^{\frac{3d}{2}}},$$
where $C'_1, C_2$ are real numbers and $C_3$ is a positive real number.
It shows that the family $(p_{k,j;u}^{(n)}(0,x))$ is summable, from which
we deduce that the Green function $G_{k,j;u}(0,x)$ is finite.
Recall that $G_{k,j;u}(x,y)=G_{k,j}(x,y)\mathrm{e}^{u\cdot (y-x)}$, so that the Green function $G_{k,j}(0,x)$ also is finite.
\end{proof}

We now deduce from Propositions~\ref{theorem1spitzer} and \ref{theorem1'spitzer} asymptotic estimates of the Green function.
If $v\in \R^d$, we denote by $\langle v\rangle$ the closest $\Z^d$-vector from $v$.
Actually, if one of the coordinates of $v$ is exactly of the form $m+1/2$, $m$ being an integer, the choice of $\langle v\rangle$ does not matter for what we intend to do.
For example, we choose $m$ for the corresponding coordinate of $\langle v\rangle$.
We will focus on vectors $\langle t\nabla \lambda (u)\rangle$, where $t$ is some real number.

\begin{prop}\label{theorem2spitzer}
Let $p$ be a strongly irreducible transition kernel on $\Z^d\times \{1,...,N\}$ which is $\Z^d$-invariant and satisfies Assumptions~\ref{hyp1} and \ref{hyp2}.
Recall that $G_{k,j;u}$ is the Green function associated to $p_u$.
Then, for $x\in \Z^d$ and for $u\in H$,
$$(2\pi t)^{\frac{d-1}{2}}G_{k,j;u}(x,\langle t\nabla \lambda (u)\rangle)\underset{t\rightarrow +\infty}{\longrightarrow}\frac{C(u)_k\nu(u)_j}{\sqrt{|Q_u|\Sigma_u(\nabla \lambda(u))}}.$$
When $x$ is fixed, the convergence is uniform in $u\in H$.
\end{prop}

\begin{proof}
This is an avatar of Theorem~2.2 from \cite{NeySpitzer}.
In their proof, the authors only use the asymptotic they find in their Theorem~2.1, which is analogous to Proposition~\ref{theorem1'spitzer}.
Actually, what they really prove is the following.
\begin{lem}\label{lemmaNeySpitzer}
Let $p_n(x)$ be a sequence of real numbers, depending on $x\in \Z^d$.
Let $u$ be a parameter that lies in some compact set $H$ and let $\alpha_u\in \R$, $\beta_u \in \R^d$,$\beta_u\neq0$ depend continuously on $u$. Let $\Sigma_u$ be a positive definite quadratic form that depends continuously on $u$. Define
$$a_n(x,u,\gamma)=\left ( \frac{\|x-n\beta_u\|}{\sqrt{n}} \right )^{\gamma}\left ((2\pi n)^{\frac{d}{2}}p_n(x)-\alpha_u\mathrm{e}^{-\frac{1}{2n}\Sigma_u(x-n\beta_u)}\right ).$$
and denote by $g(x)$ the sum over $n$ of the $p_n(x)$.
If $a_n$ converges to 0, uniformly in $x\in \Z^d$, $u\in H$ and $\gamma\in [0,2d]$, then, for $x\in \Z^d$ and for $u\in H$,
$$(2\pi t)^{\frac{d-1}{2}}g(\langle t\beta_u\rangle-x)\mathrm{e}^{\langle t\beta(u)\rangle-x}\underset{t\rightarrow +\infty}{\longrightarrow}\frac{\alpha_u}{\Sigma_u(\beta_u)}.$$
When $x$ is fixed, the convergence is uniform in $u\in H$.
\end{lem}
Using this lemma for every entry $G_{k,j;u}$ of the Green matrix, and using Proposition~\ref{theorem1'spitzer}, we deduce Proposition~\ref{theorem2spitzer}.
\end{proof}

We now describe the Martin boundary of our chain, using Proposition~\ref{theorem2spitzer}.
Fix some point $(x_0,k_0)\in \Z^d\times \{1,...,N\}$.
Recall that the Martin kernel $K$ is defined as a quotient of two Green functions.
Here, for $(x,k),(y,j)\in \Z^d\times \{1,...,N\}$, we have
$$K((x,k),(y,j))=\frac{G((x,k),(y,j))}{G((x_0,k_0)(y,j))}=\frac{G_{k,j}(x,y)}{G_{k_0,j}(x_0,y)}=:K_{k,j}(x,y).$$
Also recall that we have a homeomorphism given by
$$u\in H\mapsto \frac{\nabla \lambda (u)}{\|\nabla \lambda (u)\|}\in \Ss^{d-1}.$$

Denote by $\partial (\Z^d\times \{1,...,N\})$ the $\mathrm{CAT}(0)$ boundary of $\Z^d\times \{1,...,N\}$ (see Section~\ref{SectionGeometricboundaries} for the precise definition).
Recall that a sequence $((x_n,k_n))$ converges to a point $\tilde{x}$ if
\begin{itemize}[label=$\cdot$]
    \item either $(x_n)$ converges to $x$ in $\Z^d$ and $(k_n)$ converges to $k$ in $\{1,...,N\}$, in which case $\tilde{x}=(x,k)$,
    \item or $(x_n)$ converges to $x'$ in the $\mathrm{CAT}(0)$ boundary of $\Z^d$, in which case $\tilde{x}=x'$.
\end{itemize}
In particular, $\partial (\Z^d\times \{1,...,N\})=\partial \Z^d$. It is a sphere at infinity which does not depend on the thickening  $\{1,...,N\}$.

We thus have a homeomorphism $\varphi:\tilde{x}\in \partial (\Z^d\times \{1,...,N\})\mapsto u\in H$, where $u$ is the point in $H$ such that
$$\tilde{x}\in \Ss^{d-1}=\frac{\nabla \lambda (u)}{\|\nabla \lambda (u)\|}.$$

We show that $\partial (\Z^d\times \{1,...,N\})$ is the Martin boundary of our chain.
Precisely, we have the following.

\begin{prop}\label{latticeMartinboundary}
Let $p$ be a strongly irreducible transition kernel on $\Z^d\times \{1,...,N\}$ which is $\Z^d$-invariant and satisfies Assumptions~\ref{hyp1} and \ref{hyp2}.
If $y_n\in \Z^d$ converges to $\tilde{y}\in \partial\Z^d$, let $u=\varphi(\tilde{y})$.
Then, for every $x\in \Z^d$ and for every $k,j\in \{1,...,N\}$,
$K_{k,j}(x,y_n)$ converges to $\frac{C(u)_k}{C(u)_{k_0}}\mathrm{e}^{u\cdot (x-x_0)}$.
\end{prop}

Denote by $K((x,k),\tilde{y})$ the extension of the Martin kernel thus defined, that is
$$K((x,k),\tilde{y})=\frac{C(u)_k}{C(u)_{k_0}}\mathrm{e}^{u\cdot (x-x_0)}.$$
Notice that the limit does not depend on $j$, and so, it does not depend on the level on which $(y_n,j)$ asymptotically lies.
It shows the Martin boundary does not depend on the thickening.
When a sequence $(y_n,j_n)$ tends to infinity, the Martin kernel does not record the changes of levels of $j_n$, but only the asymptotic direction of $y_n$.

\begin{proof}
Recall that the chain is strongly irreducible.
Let $(y_n)$ be a sequence of $\Z^d$ which converges to $\tilde{y}\in \partial\Z^d$.
In particular, $\theta_n=\frac{y_n}{\|y_n\|}$ is well defined, up to taking $n$ large enough, since $y_n$ tends to infinity.
Denote by $u_n\in H$ the corresponding point.
Since $y_n$ converges to $\tilde{y}$, $\theta_n$ converges to $\theta$ and so $u_n$ converges to $u$.
Using continuity, $\lambda(u_n)$ converges to $\lambda(u)$, $Q_{u_n}$ to $Q_u$ and so $\Sigma_{u_n}$ converges to $\Sigma_u$ and $|Q_{u_n}|$ to $|Q_u|$.
Finally, $\|\nabla\lambda(u_n)\|$ converges to $\|\nabla\lambda(u)\|$.

Let $x\in \Z^d$.
Recall that $G_{k,j;u_n}(x,y_n)=G_{k,j}(x,y_n)\mathrm{e}^{u_n\cdot (y_n-x)}$.
Besides, defining
$$t_n=\frac{\|y_n\|}{\|\nabla \lambda (u_n)\|},$$
$t_n\in \R$ and $t_n$ tends to infinity.
Furthermore, $y_n=\langle t_n\nabla \lambda (u_n)\rangle$.

From Proposition~\ref{theorem2spitzer}, we deduce that
$$\frac{G_{k,j}(x,y_n)}{G_{k,j}(x_0,y_n)} \mathrm{e}^{u_n\cdot(x_0-x)}\underset{n\rightarrow \infty}{\longrightarrow}\frac{C(u)_k}{C(u)_{k_0}},$$
that is
\begin{equation*}
    K_{k,j}(x,y_n)\underset{n\rightarrow \infty}{\longrightarrow}\frac{C(u)_k}{C(u)_{k_0}}\mathrm{e}^{u\cdot (x-x_0)}.\qedhere
\end{equation*}
\end{proof}

We now show that $\Z^d\times \{1,...,N\}$ separates points on the boundary.
\begin{prop}\label{theoremseparation}
If $\tilde{y}_1\neq \tilde{y}_2$ are two points on the boundary, then there exists a sequence $(x_n,k_n)\in \Z^d\times \{1,...,N\}$ such that $K((x_n,k_n),\tilde{y}_1)$ tends to infinity and
$K((x_n,k_n),\tilde{y}_2)$ converges to 0.
In particular, there exists $(x,k)\in \Z^d\times \{1,...,N\}$ such that $K((x,k),\tilde{y}_1)\neq K((x,k),\tilde{y}_2)$
\end{prop}

\begin{proof}
Define $u_1\in H$ that corresponds to $\tilde{y}_1$ and $u_2$ that corresponds to $\tilde{y}_2$, so that
$$K((x,k),\tilde{y}_i)=\frac{C(u)_k}{C(u)_{k_0}}\mathrm{e}^{u_i\cdot (x-x_0)}.$$

Since $\tilde{y}_1\neq \tilde{y}_2$, we have $u_1\neq u_2$.
We can then find $\theta \in \Ss^{d-1}$ such that $\theta\cdot u_1>0$ and $\theta\cdot u_2< 0$.
Then, if $x_n$ is a sequence of $\Z^d$ converging in direction to $\theta$ and if $k\in \{1,...,N\}$,
$K((x_n,k),\tilde{y}_1)$ tends to infinity, whereas $K((x_n,k),\tilde{y}_2)$ converges to 0.
We can then find $n$ such that those two quantities are not the same.
\end{proof}

We can now summarize all the technical results of this section into the following proposition.
\begin{prop}\label{propsummarize}
Let $p$ be a strongly irreducible transition kernel on $\Z^d\times \{1,...,N\}$ which is $\Z^d$-invariant and satisfies Assumptions~\ref{hyp1} and \ref{hyp2}.
Then, the Martin compactification coincides with the $\mathrm{CAT}(0)$ compactification.
\end{prop}

The following lemma shows that if the chain is only irreducible, then we can reduce to the case of strong irreducibility.
It was already used in similar contexts in \cite{NeySpitzer} and \cite{Spitzer} (see \cite[Proposition~26.1]{Spitzer}).

If $p$ is a chain on a countable space $E$, define the modified chain $\tilde{p}$ by
\begin{equation}\label{modifiedchain}
    \tilde{p}(x,y)=(1-\alpha)\delta(x,y)+\alpha p(x,y),
\end{equation}
where $\delta(x,y)=0$ if $x\neq y$ and 1 otherwise and where $0<\alpha <1$ is fixed.
Denote by $\tilde{p}^{(n)}$ the $n$th convolution power of $\tilde{p}$, with $\tilde{p}^{(0)}=\delta(x,y)$. Also denote by $\tilde{G}$ the associated Green function:
$$\tilde{G}(x,y)=\sum_{n\geq 0}\tilde{p}^{(n)}(x,y).$$

\begin{lem}\label{Spitzertrick}
With these notations, $\alpha \tilde{G}(x,y)=G(x,y)$ and thus the Martin kernels are the same.
\end{lem}
As announced, it was already used in \cite{NeySpitzer} and \cite{Spitzer} but without a proof.
We provide one for the convenience of the reader.
\begin{proof}
By induction, one shows that
$$\tilde{p}^{(n)}(x,y)=\sum_{k=0}^n\begin{pmatrix}n \\ k\end{pmatrix} (1-\alpha)^{n-k}\alpha^kp^{(k)}(x,y).$$
Thus,
\begin{align*}
    \tilde{G}(x,y)&=\sum_{n\geq 0}\sum_{k=0}^n\begin{pmatrix}n \\ k\end{pmatrix} (1-\alpha)^{n-k}\alpha^kp^{(k)}(x,y)\\
    &=\sum_{k\geq 0}p^{(k)}(x,y)\sum_{n\geq k}\begin{pmatrix}n \\ k\end{pmatrix} (1-\alpha)^{n-k}\alpha^k.
\end{align*}

What is left to show is that
$$\sum_{n\geq k}\begin{pmatrix}n \\ k\end{pmatrix} (1-\alpha)^{n-k}\alpha^k=\frac{1}{\alpha}.$$
Write
$$\sum_{n\geq k}\begin{pmatrix}n \\ k\end{pmatrix} (1-\alpha)^{n-k}\alpha^k=\frac{\alpha^k}{k!}\sum_{n\geq k}\frac{n!}{(n-k)!}(1-\alpha)^{n-k}.$$
Recognize the $k$th derivative of the function
$x\mapsto \frac{1}{1-x},$
so that
\begin{equation*}\sum_{n\geq k}\begin{pmatrix}n \\ k\end{pmatrix} (1-\alpha)^{n-k}\alpha^k=\frac{\alpha^k}{k!}\left (\frac{1}{1-x}\right )^{(k)}(1-\alpha)=\frac{1}{\alpha}.\qedhere \end{equation*}
\end{proof}

Replacing $p$ with $\tilde{p}$, we can assume that $p(0,0)>0$.
If $p$ is irreducible and satisfies $p(0,0)>0$, then $p$ is strongly irreducible.
We can thus state the following.

\begin{theorem}
Let $p$ be an irreducible transition kernel on $\Z^d\times \{1,...,N\}$ which is $\Z^d$-invariant and such that the strongly irreducible chain $\tilde{p}$ given by~(\ref{modifiedchain}) satisfies Assumptions~\ref{hyp1} and \ref{hyp2}.
Then, the Martin compactification coincides with the $\mathrm{CAT}(0)$ compactification.
\end{theorem}

Finally, let us say a few words about finitely generated virtually abelian groups.
Let $\Gamma$ be such a group.
There is a subgroup of $\Gamma$ isomorphic to $\Z^d$ and with finite index.
Denote by $L$ the quotient $\Gamma/\Z^d$, which is a finite set.
Any section $L\rightarrow \Gamma$ provides an identification between $\Gamma$ and a set $\Z^d\times \{1,...,N\}$ and the $\mathrm{CAT}(0)$ compactification of $\Gamma$ does not depend on the choices of the abelian subgroup and the section $L\rightarrow \Gamma$.
Since we were able to describe the Martin boundary of $\Z^d\times L'$, when $L'$ is finite, whether it is a group or not, we have the following.

\begin{prop}\label{generalabeliangroups}
Let $\Gamma$ be a finitely generated virtually abelian group.
Let $p$ be an irreducible transition kernel on $\Gamma \times \{1,...,N\}$
and assume that if $p$ is seen as a chain on $\Z^d\times \{1,...,N'\}$, the strongly irreducible chain $\tilde{p}$ given by~(\ref{modifiedchain})satisfies Assumptions~\ref{hyp1} and \ref{hyp2}.
Then, the Martin compactification coincides with the $\mathrm{CAT}(0)$ compactification.
\end{prop}

\section{Markov chains on a thickened lattice}\label{SectionMarkovchainsonathickenedlattice}

We now show Theorems~\ref{theorem2} and \ref{theorem3}.
We will actually prove that the Martin boundary is minimal in Theorem~\ref{theorem2} in Section~\ref{SectionMinimalMartinboundary} and we focus here on showing that the Martin boundary coincides with the geometric boundary.

We consider a $\Z^d$-invariant Markov chain on $\Z^d\times \{1,...,N\}$.
It is defined by the transition kernel
$$p((x,k),(y,j))=p_{k,j}(x,y)=p_{k,j}(0,y-x),x,y\in \Z^d,k,j\in \{1,...,N\},$$
$$\forall k\in \{1,...,N\}, \sum_{x\in \Z^d}\sum_{1\leq j\leq N}p_{k,j}(0,x)= 1.$$
We assume that $p$ is strongly irreducible, which is enough to prove Theorems~\ref{theorem2} and \ref{theorem3} according to Lemma~\ref{Spitzertrick}.
To $p$, we associate the matrix $F(u),u\in \R^d$ as previously, whose entries are defined by
$$F_{k,j}(u)=\sum_{x\in \Z^d}p_{k,j;u}(0,x)=\sum_{x\in \Z^d}p_{k,j}(0,x)\mathrm{e}^{u\cdot x}.$$
Denote by $\nu(u)$ and $C(u)$ left and right eigenvectors for $F(u)$ as previously.
Since $p$ is Markov, $F(0)$ is stochastic, so that $C(0)=(c,c,...,c),c\in \R,c\neq 0$.
As $\sum_j\nu(0)_j=1$ by assumption, $c=1$.

The vector $\nu(0)$ coincides with the vector $\nu_0$ defined in the introduction.
Assume that the chain has finite support and is non-centered, i.e.
$$\sum_{x\in \Z^d}\sum_{k,j}\nu(0)_kxp_{k,j}(0,x)\neq0.$$

Under these assumptions, we prove that we can directly apply Proposition~\ref{propsummarize}.
Indeed, according to Lemma~\ref{lemmafinitesupport}, Assumptions~\ref{hyp1} holds.
Furthermore, the following holds.

\begin{lem}\label{noncentered}
Assumption~\ref{hyp2} holds if and only if the Markov chain is non-centered.
More precisely, recalling the definition
$$\overrightarrow{p}=\sum_{x\in \Z^d}\sum_{k,j}\nu(0)_kxp_{k,j}(0,x),$$
one has
$$\nabla \lambda (0)=\overrightarrow{p}.$$
\end{lem}

\begin{proof}
Since $F(0)$ is stochastic, $\lambda(0)=1$, i.e.\ $0\in H$.
We also have
$$\nabla F_{k,j}(0)=\sum_{x\in \Z^d}xp_{k,j}(0,x).$$
Also recall that $\nabla \lambda(u)=\nu(u)\nabla F(u)C(u)$.
Thus,
$$\nabla \lambda(0)=\sum_{x\in \Z^d}\sum_{k,j}\nu(k)xp_{k,j}(0,x).$$

We deduce from this equality that $\nabla \lambda (0)\neq 0$ if and only if the Markov chain is non-centered.
Since $\lambda$ is strictly convex, the minimum of $\lambda$ is not reached at 0 if and only if $\nabla \lambda (0)\neq 0$, i.e.\ 1 is not the minimum of $\lambda$ if and only if the Markov chain is non-centered.
\end{proof}

The two assumptions of the last section are then satisfied and one can apply Proposition~\ref{propsummarize}.
Now, Lemma~\ref{Spitzertrick} shows that we can reduce to the case of a strongly irreducible Markov chain, so
we deduce from this Theorem~\ref{theorem2}.
\qed

We now show that in the centered case, the Martin boundary is trivial.
\begin{prop}\label{trivialMartincentered}
Consider an irreducible, $\Z^d$-invariant, finitely supported Markov chain on the thickened lattice $\Z^d\times \{1,...,N\}$ and assume that it is centered.
Then, the Martin compactification coincides with the one-point compactification.
\end{prop}

\begin{proof}
For a random walk on the abelian group, positive minimal harmonic functions are of the form
$$x\in \Z^d\mapsto \mathrm{e}^{u\cdot x},$$ with
$$\sum_{x\in \Z^d}p(0,x)\mathrm{e}^{u\cdot x}=1$$
(see for example \cite[Theorem~7.1]{Sawyer}).
In our situation, the same arguments as in \cite{Sawyer} show that positive minimal harmonic functions in $\Z^d\times \{1,...,N\}$ are of the form
$$(x,k)\in \Z^d\times \{1,...,N\}\mapsto C(u)_k\mathrm{e}^{u\cdot x},$$ with
$$\sum_k\sum_{x\in \Z^d}p_{k,j}(0,x)\mathrm{e}^{u\cdot x}C(u)_j=1,$$
which exactly means that $C(u)$ is a right eigenvector for $F(u)$ with eigenvalue $\lambda(u)=1$.
This was actually the reason for introducing the matrix $F$.
If the Markov chain on $\Z^d\times \{1,...,N\}$ is centered, then $\nabla\lambda(0)=0$, according to Lemma~\ref{noncentered}, which means that $\lambda(0)$ is the minimum of $\lambda$.
Thus, the value 1 is only reached at 0 and there is only one positive minimal harmonic function up to multiplication by a constant.
Since every positive harmonic function is a linear combination of positive minimal harmonic functions (see Section~\ref{SectionMarkovchainsandMartinboundary}),
we see that there is only one positive harmonic function up to multiplication by a constant.
In particular, the Martin boundary is trivial.
\end{proof}

We now prove Theorem~\ref{theorem3}.
Thanks to Lemma~\ref{Spitzertrick}, we can assume that the chain is strongly irreducible.
First, if the random walk is non-centered, then Assumptions~\ref{hyp1} and \ref{hyp2} from the last section are satisfied.
We can then apply Proposition~\ref{corotransience} to conclude.
On the contrary, assume that $\nabla \lambda (0)=0$.
Since the chain is strongly irreducible,
we can use Proposition~\ref{theorem1spitzer}.
For $u=0$, we get that
$(2\pi n)^{\frac{d}{2}}P^{(n)}(0,x)-|Q_0|^{-\frac{1}{2}}\mathrm{e}^{-\frac{1}{2n}\Sigma_0(x)}C(0)\cdot \nu(0)$
uniformly converges to 0.
Since $\mathrm{e}^{-\frac{1}{2n}\Sigma_0(x)}$ converges to 1,
$P_{k,j}^{(n)}(0,x)$ is equivalent to $C n^{-\frac{d}{2}}$, where $C$ is a positive real number.
This quantity is summable if and only if $d\geq 3$.
\qed

We now deduce from Proposition~\ref{TCL} a central limit theorem in the context of Markov chains on $\Z^d\times \{1,...,N\}$.
First, we can adapt Lévy's continuity theorem (see \cite[Theorem~26.3]{Billingsley}) to prove the following lemma.
Let $\overline{X}=(X,K)$ be a random variable in $\R^d\times \{1,...,N\}$ and denote by $\mu_j$ the restriction of the law of $\overline{X}$ to the level $K=j$, that is
$$\mathds{P}(X\in A\cap K=j)=\int_{A}\mathrm{d}\mu_j(x),\text{ for } A \text{ a Borel set in }\R^d.$$
Notice that $\mu_j$ is not a probability measure.
Then, $\mu=\sum \mu_j$ is the law of $X$.

Define the characteristic vector $\psi(\xi)$ of $\overline{X}$ as the vector of $\R^N$ whose $j$th coordinate is given by
$$\psi(\xi,j)=\int_{\R^d}\mathrm{e}^{ix\cdot \xi}\mathrm{d}\mu_j(x).$$
If $\overline{\mu}$ is the law of $\overline{X}$, then
$$\psi(\xi,j)=\int_{\R^d}\mathrm{e}^{ix\cdot \xi}\mathds{1}_{k=j}\mathrm{d}\overline{\mu}(x,k).$$
\begin{lem}\label{ThickenedLevy}
Let $\overline{X}_n=(X_n,K_n)$ a sequence of random variables in $\R^d\times \{1,...,N\}$ and let $\overline{X}=(X,K)$ be a random variable in $\R^d\times \{1,...,N\}$.
Denote by $\psi_n(\xi)$ and $\psi(\xi)$ the characteristic vectors of $\overline{X}_n$ and $\overline{X}$.
Then, $\overline{X}_n$ converges to $\overline{X}$ in law if and only if $\psi_n(\xi)$ converges pointwise to $\psi(\xi)$, meaning that for every $j\in \{1,...,N\}$,
$\psi_n(\xi,j)$ converges pointwise to $\psi(\xi,j)$.
\end{lem}

The lemma readily follows from Lévy's theorem on each level.
Now, $p$ is a probability transition kernel, so that $\lambda(0)$ is well defined and $C(0)$ is of the form $(1,...,1)$.
Denote by $\nu_j$ the $j$th coordinate of the left eigenvector $\nu(0)$.
Also denote by $\overline{X}_{n}^{(k)}=(X_{n}^{(k)},K_{n}^{(k)})$ the law of the Markov chain at time $n$, starting at $(0,k)$.
Direct calculation shows that $\left [\chi_0(\frac{\xi}{\sqrt{n}})\right ]^{(n)}_{k,j}$ is the $j$th coordinate of the characteristic vector of the random variable $\overline{Y}_n^{(k)}=(Y_n^{(k)},K_n^{(k)})$, where $Y_n^{(k)}=\frac{X_{n}^{(k)}-n\nabla\lambda(0)}{\sqrt{n}}$.

Also denote by $\tilde{Q}_0$ the symmetric matrix associated to the (positive definite) quadratic form $Q_0$, by $\tilde{\Sigma}_0$ the inverse of $\tilde{Q}_0$ and by $\Sigma_0$ the quadratic form associated to $\tilde{\Sigma}_0$.
Let $X$ be a random variable in $\R^d$ following the Gaussian law associated to $\Sigma_0$, that is
$$\mathds{E}[f(X)]=\int_{\R^d}f(x)\mathrm{e}^{-\frac{1}{2}\Sigma_0(x)}\mathrm{d}x.$$
Again, direct calculation shows that $\mathrm{e}^{-\frac{1}{2}Q_u(\xi)} \nu_j$ is the $j$th coordinate of the characteristic vector of the random variable $\overline{X}=(X,\nu)$, whose law is defined by
$$\mathds{E}[f(\overline{X})]=\sum_j\nu_j\int_{\R^d}f((x,j))\mathrm{e}^{-\frac{1}{2}\Sigma_0(x)}\mathrm{d}x.$$
Thus, the random variable $\overline{X}$ follows an averaged Gaussian law.

\begin{theorem}\label{TCL2}
If the chain is strongly irreducible and has an exponential moment,
then $\overline{Y}_n^{(k)}$ converges to $\overline{X}$ in law.
\end{theorem}

\begin{proof}
If the chain is strongly irreducible and has an exponential moment, then $0\in H$ and we can apply Proposition~\ref{TCL} and Lemma~\ref{ThickenedLevy} to conclude.
\end{proof}

\begin{rem}
The limit law is independent of the level $k$ of the starting point.
\end{rem}

This convergence in law is true if the chain has some finite exponential moment and if it is strongly irreducible.
With other techniques, one could prove convergence in law under finiteness of second moments and weaken the assumption of strong irreducibility to irreducibility.


\section{Random walks in free products}\label{SectionRandomwalksinfreeproducts}
In this section, we prove Theorem~\ref{theorem1}.
Again, we will only prove here that the Martin boundary coincides with the geometric boundary and we will show that the boundary is minimal in the next section.
We consider an irreducible random walk on the free product $\Z^{d_1}\star \Z^{d_2}$.
Actually, if $d_1=d_2=1$, then $\Z^{d_1}\star \Z^{d_2}$ is nothing else than the free group of rank 2.
In this particular case, the theorem was proved by Y.~Derriennic in \cite{Derriennic}.
The proof we give below also works in this setting, but it becomes much simpler.

Let $e$ be the neutral element of $\Z^{d_1}\star \Z^{d_2}$.
Recall that an element $g$ of $\Z^{d_1}\star \Z^{d_2}$ that differs from $e$ can be uniquely written as
$g=a_1b_1...a_nb_n$, with $a_i\in \Z^{d_1}$, $b_i\in \Z^{d_2}$ and $a_i\neq 0$ except maybe $a_1$, $b_i\neq 0$ except maybe $b_n$.
We say that the sequence $a_1,b_1,...,a_n,b_n$ represents $g$.
We call \emph{length} of $g$ (or distance between $g$ and $e$) the quantity
$$|g|:=\|a_1\|_1+\|b_1\|_1+...+\|a_n\|_1+\|b_n\|_1,$$
where $\|v\|_1=|v_1|+...+|v_d|$, if $v=(v_1,...,v_d)$.
We can define a metric on $\Z^{d_1}\star \Z^{d_2}$ declaring that $d(g,h)=|g^{-1}h|$. We call this metric the \emph{word metric}
Finally, an \emph{infinite word} is an infinite sequence $a_1,b_1,...,a_n,b_n,...$ alternating elements of $\Z^{d_1}$ and elements of $\Z^{d_2}$, such that every one of them, except maybe $a_1$ differs from the neutral element.

The transition kernel of the Markov chain can now be written as
$$p(g,h)=p(e,g^{-1}h)=\mu(g^{-1}h),$$
where $\mu$ is a probability measure on $\Z^{d_1}\star \Z^{d_2}$.
The equality $p(g,h)=p(e,g^{-1}h)$ means that the Markov chain is invariant under translation: it is a random walk.
Finite support for the random walk means finite support for $\mu$.
Let $r(\mu)$ be the supremum of $|g|$ over $g\in \Z^{d_1}\star \Z^{d_2}$ such that $\mu(g)>0$.
It is the minimal radius of a ball in which the support of $\mu$ is included.

\bigskip
\emph{Thanks to Lemma~\ref{Spitzertrick}, we can and will assume in all this section that the random walk is strongly irreducible and satisfies $\mu(e)>0$.}

\subsection{Transitional sets}\label{Sectiontransitionalsets}
We adapt to our situation the notion of transitional set, as defined in \cite{Derriennic}.
\begin{definition}
Let $g\in \Z^{d_1}\star \Z^{d_2}$ be an element represented by $a_1,b_1,...,a_n,b_n$.
The size of $g$ is the number of non-zero elements among $a_1,...,b_n$.
It is denoted by $s(g)$.
\end{definition}

In other words, the size thus defined is the number of changes of $\Z^{d_i}$ factors.
It differs from the length of a geodesic from $e$ to $g$.

\begin{definition}
Let $g=a_1b_1...a_nb_n\in \Z^{d_1}\star \Z^{d_2}$.
If $p\leq s(g)$, the prefix of size $p$ of $g$ is the element $h=a_1b_1...a_kb_k\in\Z^{d_1}\star \Z^{d_2}$,
where $k\leq n$ and with $s(h)=p$.
\end{definition}

We now define transitional sets.
\begin{definition}\label{deftransitionalset}
If $g\neq h\in \Z^{d_1}\star \Z^{d_2}$, write $g^{-1}h=a_1b_1...a_nb_n$.
Let $x=a_1b_1...a_kb_k$ be a prefix of $g^{-1}h$ ($k\leq n$).
A transitional set between $g$ and $h$ is a set of the form
$$V=g\cdot B(x,r(\mu)),$$
where $B(x,r(\mu))$ is the ball of center $x$ and radius $r(\mu)$ (for the word metric).
\end{definition}

The following lemma justifies the name transitional set.

\begin{lem}\label{lemmatransitionalsets}
Let $g,h\in \Z^{d_1}\star \Z^{d_2}$.
Assume there exists a transitional set $V$ between $g$ and $h$.
The random walk starting at $g$ cannot reach $h$ before visiting $V$.
\end{lem}

For formal proof, we refer to \cite[Lemma~III.1]{Derriennic}.
We only give a heuristic explanation.
Let $x$ be a prefix of $g^{-1}h$ such that $V=g\cdot B(x,r(\mu))$.
By definition, $x$ is of the form $x=a_1b_1...a_kb_k$.
To go from $g$ to $h$, the random walks goes through some path $X_0=g,X_1,...,X_m=h$.
Let $l$ be the first time such that $gx$ is on a geodesic between $g$ and $X_l$.
If $l=0$, i.e.\ if $x$ is trivial, then $g\in V$.
Assume that $l>0$.
If $X_{l-1}$ or $X_l$ is in $V$, then the random walk visits $V$.
If not, $d(X_{l-1},X_l)>r(\mu)$, since $gx$ is one of the points where the geodesic from $g$ to $h$ changes cosets.
This is a contradiction.

\bigskip
We use notations of \cite{Derriennic}. They will be useful in the following.
Let $g\in \Z^{d_1}\star \Z^{d_2}$ and let $V\subset \Z^{d_1}\star \Z^{d_2}$ and $v\in V$.
We denote by $p_{g}^V(v)$ the probability that the first visit to $V$ of the random walk starting at $g$ is at $v$.
We can see $\{p_{g}^V(v)\}_{v\in V}$ as a vector whose coordinates are indexed by $V$.
We denote by $p_g^V$ this vector.

If $V\subset \Z^{d_1}\star \Z^{d_2}$ and $W\subset\Z^{d_1}\star \Z^{d_2}$,
we denote by $P_{V}^W$ the matrix whose lines are indexed by $V$, columns by $W$, and whose entries are
$$(P_{V}^W)_{v,w}=p_v^W(w).$$
Those matrices are sub-stochastic.

Recall that for $g,h \in \Z^{d_1}\star \Z^{d_2}$, $\mathds{P}(g\rightarrow h)$ is the probability that the random walk starting at $g$ reaches $h$.
If $g\in \Z^{d_1}\star \Z^{d_2}$ and if $V\subset \Z^{d_1}\star \Z^{d_2}$, denote by $p_{V}^g$ the vector whose coordinates are $\{\mathds{P}(v\rightarrow g)\}_{v\in V}$.
If $V$ is a transitional set between $g$ and $h$, then the random walk from $g$ to $h$ goes through $V$.
We rewrite this fact as
$$\mathds{P}(g\rightarrow h)=\sum_{v\in V}p_{g}^V(v)\mathds{P}(v\rightarrow h)=p_g^V\cdot p_V^h.$$

Assume that $V_1,...,V_{n+1}$ are $n+1$ disjoint transitional sets between $g$ and $h$.
We can order them according to the geodesic from $g$ to $h$:
$d(g,V_i)<d(g,V_{i+1})$.
Denote by $P_i$ the matrix $P_{V_i}^{V_{i+1}}$.
Then, the random walk from $g$ to $h$ has to go through each set $V_i$ successively
This translates into the following equality
$$\mathds{P}(g\rightarrow h)=p_g^{V_1}\cdot P_1\dots P_np_{V_{n+1}}^h.$$

We will now apply Theorem~\ref{PerronFrob2} to this last equality.
The matrices $P_V^W$ have non-negative entries.
If $V$ and $W$ are two transitional sets, the matrix $P_V^W$ is a square matrix of size the cardinal of the ball centered at $e$ of radius $r(\mu)$.
Denote by $k(\mu)$ this cardinal.
Then, $P_V^W$ is a $k(\mu)\times k(\mu)$ matrix.
To ensure positivity properties of this matrix, we will use the following lemma.
\begin{lem}\label{lemmaRmu}
Let $g\in \Z^{d_1}\star \Z^{d_2}$.
Let $B$ be the ball of center $g$ and radius $r(\mu)$.
There exists an integer $R(\mu)$ such that for every $h\in B$, the probability that the random walk starting at $g$ reaches $h$ before leaving the ball of center $g$ and radius $R(\mu)$ is positive.
\end{lem}

\begin{proof}
Since the random walk is invariant under translation, we can assume that $g=e$.
Since it is irreducible, for every $h \in B$, there is a path from $e$ to $h$ such that the probability of following this path, starting at $e$, is positive.
This path stays in some ball of radius $R(h)$.
It then suffices to define $R(\mu)=\mathrm{sup}\{R(h),h\in B\}$.
\end{proof}

We deduce from this lemma that
if one of the entry of a matrix $P_V^W$ is zero, then the whole column is 0, if $d(V,W)>R(\mu)$.

Let $T$ be a positive matrix whose columns are either zero or positive.
Such a matrix acts on the cone of positive vectors of $\R^{k_{\mu}}$.
Let $C_{\mu}$ be this cone and let $C'_{\mu}$ be the intersection between this cone and the unit sphere: $C'_{\mu}=C_{\mu}\cap \Ss^{k_{\mu}-1}$.
Such a matrix $T$ also acts on $C'_{\mu}$ via
$$T:v\in C'_{\mu}\mapsto \frac{Av}{\|Av\|}\in C'_{\mu}.$$
Denote by $d_{\mathcal{H}}$ the Hilbert distance on $C'_{\mu}$, as defined in Section~\ref{SectionPerronFrobenius}.
The diameter of $T$ is by definition the diameter of the set $T\cdot C'_{\mu}=\{\frac{Tv}{\|Tv\|}, v\in C'_{\mu}\}$ for the distance $d_{\mathcal{H}}$.
Denote it by $\Delta(T)$.

\begin{prop}\label{propalpha1}
Let $V,W$ be two disjoint transitional sets between two elements $g$ and $h$.
Assume that $d(V,W)>R(\mu)$.
Then, there exists a real number $0<\alpha<1$ independent of $V$ and $W$ (and of $g$ and $h$) such that for every $v,v'\in V$ and for every $w\in W$,
$P_V^W(v,w)\geq \alpha P_V^W(v',w)$.
\end{prop}

\begin{proof}
The number $P_V^W(v,w)$ is the probability that the random walk starting at $v$ first visits $W$ at $w$.
Denote by $P(v,v';R(\mu))$ the probability that the random walk starting at $v$ reaches $v'$ before leaving the ball of center $v$ and radius $R(\mu)$.
According to Lemma~\ref{lemmaRmu}, $P(v,v';R(\mu))>0$.
Furthermore, since $d(V,W)>R(\mu)$, $P_V^W(v,w)\geq P(v,v';R(\mu))P_V^W(v',w)$.

Since there is a finite number of $v$ and $v'$ inside $V$, one can get a uniform lower bound for $P(v,v';R(\mu))$.
Besides, the random walk is invariant under translation, so that this lower bound can be chosen independently of $V$.
Thus, one can find $\alpha\in (0,1)$, such that $P(v,v';R(\mu))\geq\alpha$.
\end{proof}

\begin{coro}\label{coroalpha}
Let $V$ and $W$ be two disjoint transitional sets between $g$ and $h$.
Assume that $d(V,W)>R(\mu)$.
The diameter of $P_V^W$ is bounded, independently of $V$ and $W$: there exists $\Delta>0$ such that $\Delta(P_V^W)\leq \Delta$.
\end{coro}

\begin{proof}
Let $x\in C'_{\mu}$.
The vector $x$ can be indexed by elements of $W$.
The coordinate of the vector $P_V^W(x)$ which corresponds to $v\in V$ is
$$[P_V^W(x)]_v=\sum_{w\in W}p_v^W(w)x_w.$$
Let $v_0\in V$.
Then, for every $v\in V$, $p_v^W(w)\leq \frac{1}{\alpha} p_{v_0}^W(w)$, so thanks to Proposition~\ref{propalpha1}, $[P_V^W(x)]_v\leq \frac{1}{\alpha}[P_V^W(x)]_{v_0}$.
In particular,
$$\|P_V^W(x)\|\leq \frac{\sqrt{k_{\mu}}}{\alpha}[P_V^W(x)]_{v_0}.$$

Let $y$ be the vector $\frac{P_V^W(x)}{\|P_V^W(x)\|}$ of $C'_{\mu}$, i.e. $y=P_V^W\cdot x$.
Then, $y_{v_0}=\frac{[P_V^W(x)]_{v_0}}{\|P_V^W(x)\|}$, so $y_{v_0}\geq \frac{\alpha}{\sqrt{k_{\mu}}}=\beta$.
Thus, the range of $P_V^W$ lies inside the set
$$\{y\in C'_{\mu}, y_k\geq \beta, k=1,...,k_{\mu}\}.$$
This is a compact set of $C'_{\mu}$ with respect to the Euclidean distance, so it is compact for the Hilbert distance, since both distances induce the same topology (see Remark~\ref{remarktopologies}).
\end{proof}

\subsection{Convergence of Martin kernels}\label{SectionconvergenceofMartinkernelsfreeproducts}
To study the Martin kernels $K(g,h)$, the base point is $e$.
There will be several steps for proving Theorem~\ref{theorem1}.
We first show that if $(g_n)$ converges to some infinite word, then the Martin kernel $K(\cdot, g_n)$ converges to some function.
Recall that convergence to an infinite word is defined as follows:
a sequence $(g_n)$ converges to an infinite word $\tilde{g}$ if for every $p$, there exists $n_0$ such that for every $n\geq n_0$, prefixes of size $p$ of $g_n$ and $\tilde{g}$ both exist and match.

\begin{prop}\label{freeproductMartinboundary1}
Let $(g_n)$ be a sequence in $\Z^{d_1}\star \Z^{d_2}$ and $g\in \Z^{d_1}\star \Z^{d_2}$.
Assume that $(g_n)$ converges to some infinite word $\tilde{g}$.
Then, $K(g,g_n)$ converges to some quantity that we denote by $K(g,\tilde{g})$.
\end{prop}

\begin{proof}
By definition, $K(g,g_n)=\frac{G(g,g_n)}{G(e,g_n)}=\frac{\mathds{P}(g\rightarrow g_n)}{\mathds{P}(e\rightarrow g_n)}$.
We first study $\mathds{P}(g\rightarrow g_n)$.

If $g_n$ converges to $\tilde{g}$, in particular, $s(g_n)$ tends to infinity.
Thus, one can find disjoint transitional sets $V_1,...,V_{\varphi(n)}$ between $g$ and $g_n$, such that $\varphi(n)$ tends to infinity.
One can even assume that $d(V_i,V_{i+1})>R(\mu)$ and $d(V_{\varphi(n)},g_n)>R(\mu)$.
Denote by $P_i$ the matrix $P_{V_i}^{V_{i+1}}$ for every transitional sets $V_i$ and $V_{i+1}$.
Then,
$$\mathds{P}(g\rightarrow g_n)=p_g^{V_1}\cdot P_1\dots P_{\varphi(n)-1}p_{V_{\varphi(n)}}^{g_n}.$$

Up to some finite prefix, $g_n$ and $g^{-1}g_n$ begin with the same letters.
Thus, taking $d(g,V_1)$ large enough, one can assume that $V_1,...,V_{\varphi(n)}$ are also transitional sets between $e$ and $g_n$.
Then,
$$\mathds{P}(e\rightarrow g_n)=p_e^{V_1}\cdot P_1\dots P_{\varphi(n)-1}p_{V_{\varphi(n)}}^{g_n}.$$

Combining those,
$$K(g,g_n)=\frac{p_g^{V_1}\cdot P_1\dots P_{\varphi(n)-1}p_{V_{\varphi(n)}}^{g_n}}{p_e^{V_1}\cdot P_1\dots P_{\varphi(n)-1}p_{V_{\varphi(n)}}^{g_n}}.$$
According to Theorem~\ref{PerronFrob2}, $P_1\dots P_{\varphi(n)-1}p_{V_{\varphi(n)-1}}^{g_n}$ converges in norm to some vector $f$, which may depend on the sequence $(g_n)$.
Thus, $K(g,g_n)$ converges to $\frac{p_g^{V_1}\cdot f}{p_e^{V_1}\cdot f}$.

To conclude one has to show that this limit does not depend on $(g_n)$.
Let $(g'_n)$ be another sequence converging to $\tilde{g}$.
When $n$ tends to infinity, $g_n$ and $g'_n$ have a common prefix of arbitrarily large size.
Then, one can choose transitional sets for both $g_n$ and $g'_n$, so that the matrices $P_i$ for $g_n$ and $g'_n$ are the same.
According to Corollary~\ref{coroalpha} and Theorem~\ref{PerronFrob2}, the limits of $P_1\dots P_{\varphi(n)-1}\cdot p_{V_{\varphi(n)-1}}^{g_n}$ and $P_1\dots P_{\varphi(n)-1}\cdot p_{V_{\varphi(n)-1}}^{g'_n}$ are the same.

Since neither $f$ nor $V_1$ depend on the sequence $(g_n)$, the limit $K(g,g_n)$ does not depend on it, which concludes the proof.
\end{proof}

We now deal with sequences $(g_n)$ converging to infinity in some $\Z^{d_i}$ factor.
Recall that such a sequence converges in a $\Z^{d_1}$ factor $a_1b_1...a_kb_k\Z^{d_1}$ if there exists $n_0$ such that
for every $n\geq n_0$, $g_n$ has a prefix of the form $a_1b_1...a_kb_ka_{k+1,n}$ with $a_{k+1,n}\in \Z^{d_1}$ converging in the $\mathrm{CAT}(0)$ boundary of $\Z^{d_1}$.
That is, $\|a_{k+1,n}\|$ tends to infinity and $\frac{a_{k+1,n}}{\|a_{k+1,n}\|}$ converges to some point on the sphere. Denote by $\theta$ this point on the sphere and
denote by $a_1b_1...a_kb_k\theta$ the limit of $(g_n)$.
Convergence in a $\Z^{d_2}$ factor is defined similarly.

\begin{prop}\label{freeproductMartinboundary2}
Let $(g_n)$ be a sequence in $\Z^{d_1}\star \Z^{d_2}$ and $g\in \Z^{d_1}\star \Z^{d_2}$.
Assume that $(g_n)$ converges in a $\Z^{d_1}$ factor to a limit
$\tilde{g}=a_1b_1...a_kb_k\theta$.
Then, $K(g,g_n)$ converges to some quantity that we denote by $K(g,\tilde{g})$.
\end{prop}

\begin{proof}
There are essentially two cases.
Either $g_n$ stays in a bounded neighborhood of $a_1b_1...a_kb_k\Z^{d_1}$, or it leaves every such neighborhood.

Let us begin with the first case.
Let $k_1$ be an integer and denote by $E_{k_1}$ elements $\gamma$ of $\Z^{d_1}\star \Z^{d_2}$ that belong to the $k_1$-neighborhood of $a_1b_1...a_kb_k\Z^{d_1}$.
Every such element can be written as
$\gamma=a_1b_1...a_kb_kah$,
where $a\in \Z^{d_1}$ and $h$ either is trivial or is a word which starts with a non-trivial element of $\Z^{d_2}$ and which lies in the ball of center $e$ and radius $k_1$.
Actually, $a_1b_1...a_kb_ka$ is the projection of $\gamma$ onto the factor $a_1b_1...a_kb_k\Z^{d_1}$.
The ball of radius $k_1$, centered at $e$ is finite. Denote by $h_1,...,h_N$ the elements in this ball that either are trivial or begin with a non-trivial element of $\Z^{d_2}$. We can thus identify $\gamma=a_1b_1...a_kb_kah_j$ with $(a,j)$.
This gives an identification between $E_{k_1}$ and $\Z^{d_1}\times \{1,...,N\}$.

Since $g_n$ stays in some bounded neighborhood of $a_1b_1...a_kb_k\Z^{d_1}$, there exists $k_1$ such that for every $n$, $g_n\in E_{k_1}$.
Thus, $(g_n)$ can be seen as a sequence of $\Z^{d_1}\times \{1,...,N\}$ that converges in the $\mathrm{CAT}(0)$ boundary of $\Z^{d_1}\times \{1,...,N\}$.
Also, picking $k_1$ large enough, assume that $e\in E_{k_1}$ and $g \in E_{k_1}$ for simplicity.

If $\gamma,\gamma'\in E_{k_1}$, denote by $\tilde{p}(\gamma,\gamma')$ the probability that the random walks starting at $\gamma$ first returns to $E_{k_1}$ at $\gamma'$.
This defines a sub-Markov chain on $\Z^{d_1}\times \{1,...,N\}$.

\begin{lem}\label{lemmasubmarkov}
If $k_1$ is large enough, this chain is strongly irreducible and satisfies Assumptions~\ref{hyp1} and \ref{hyp2} from Section~\ref{SectionSub-Markovchainsonathickenedlattice}.
Furthermore, it is strictly sub-Markov.
\end{lem}
\begin{proof}
First, the random walk on $\Z^{d_1}\star \Z^{d_2}$ is irreducible.
If $\gamma,\gamma'\in E_{k_1}$, there exists some path $\gamma_1=\gamma,...,\gamma_l=\gamma'$ in $\Z^{d_1}\star \Z^{d_2}$ such that the probability that the random walk follows this path is positive.
Exclude from this path every $\gamma_j$ that does not lie in $E_{k_1}$.
If one could go from $\gamma_j$ to $\gamma_{j'}$ visiting $\gamma_{j_1},...,\gamma_{j_m}$ and if $\gamma_{j_1},...,\gamma_{j_m}$ were excluded, then the probability of first returning to $E_{k_1}$ at $\gamma_{j'}$, starting at $\gamma_j$ is positive.
This provides a path from $\gamma$ to $\gamma'$ that has positive probability with respect to $\tilde{p}$,
so that the new chain is again irreducible and since we assumed that $\mu(e)>0$, it is strongly irreducible.

Secondly, the random walk on $\Z^{d_1}\star \Z^{d_2}$ has finite support.
Let us now show that $\tilde{p}$ still has finite support, if $k_1$ is large enough.
If one leaves some $\Z^{d_i}$ factor at some point $\gamma$, a geodesic going back to this factor has to go through $\gamma$.
However, the random walk is not nearest neighbor, and one could avoid $\gamma$ with big jumps.
We can prevent this from happening assuming that $k_1\geq r(\mu)$, where $r(\mu)$ is the radius of the support of the random walk.
Indeed, if at time $n$, the random walk is not in $E_{k_1}$ and if $\gamma$ is its projection onto the $\Z^{d_i}$ factor, then at time $n+1$, its projection still is $\gamma$.
Thus, $\tilde{p}$ has finite support and we deduce from Lemma~\ref{lemmafinitesupport} that Assumption~\ref{hyp1} is satisfied.

Finally, let $\gamma$ be on the boundary of $E_{k_1}$, i.e. $\gamma=a_1b_1...a_kb_kah$ with $h$ beginning with a non-trivial element of $\Z^{d_2}$ and lying on the sphere of center $e$ and radius $k_1$.
The random walk on $\Z^{d_1}\star \Z^{d_2}$ is transient since the group is non-amenable, so that the probability that the random walk starting at $\gamma$ never goes back to $E_{k_1}$ is positive.
In other words,
$$\sum_{y\in E_{k_1}}\tilde{p}(x,y)<1,$$
which ensures that the chain is strictly sub-Markov.
Moreover, with notations of Section~\ref{SectionSub-Markovchainsonathickenedlattice}, it means that the matrix $F(0)$ is strictly sub-stochastic.
In particular, $\lambda(0)<1$ and so 1 is not the minimum of $\lambda$: Assumption~\ref{hyp2} is satisfied.
\end{proof}

Since all the trajectories from $\gamma$ to $\gamma'$ with respect to $\tilde{p}$ actually come from trajectories of the random walk on $\Z^{d_1}\star \Z^{d_2}$, the Green functions associated to the original random walk and associated to $\tilde{p}$ are the same.
Thus, we can apply Proposition~\ref{latticeMartinboundary}.
Denote by $x$ the projection of $g$ onto the $\Z^{d_1}$ factor, i.e. $g=(x,j)$ when identifying $E_{k_1}$ with $\Z^{d_1}\times \{1,...,N\}$.
Also denote by $x_0$ the projection of $e$.
Then,
$$K(g,g_n)\underset{n\rightarrow +\infty}{\longrightarrow}\alpha_N(g)\mathrm{e}^{u_N\cdot (x-x_0)},$$
where $u_N$ only depends on $\tilde{g}$ and a priori on $N$ and where $\alpha_N$ a priori depends on $N$ and on $g$.
The dependency on $N$ is the same as the dependency on $k_1$.

Let us show that the limit does not depend on $N$.
First, according to the formula in Proposition~\ref{latticeMartinboundary}, $\alpha_N(g)$ depends on $j$ but not on $x$ when writing $g=(x,j)$.
Assume that $k_1$ and $k'_1$ are two large enough integers to apply Lemma~\ref{lemmasubmarkov} and denote by $N$ and $N'$ the corresponding integers.
Since the Green functions associated to $\tilde{p}$ on $\Z^{d_1}\times \{1,...,N\}$ and to the same chain defined on $\Z^{d_1}\times \{1,...,N'\}$ are the same,
one deduces that
$\alpha_N\mathrm{e}^{u_N\cdot (x-x_0)}=\alpha_{N'}\mathrm{e}^{u_{N'}\cdot (x-x_0)}$.
Change the point $g$, or more accurately, change its projection $x$.
Then, the equality above still holds so it holds for every $x\in \Z^{d_1}$.
In particular, applying it to $x=x_0$, one gets that $\alpha_N=\alpha_{N'}$.
Thus, $u_N=u_{N'}$.
Denote by $K(g,\tilde{g})$ the quantity thus defined, which does not depend on $N$.

\medskip
We still have to deal with the second case, that is $g_n$ leaves every neighborhood of $a_1b_1...a_kb_k\Z^{d_1}$.
Assume that $g_n=a_1b_1...a_kb_ka_nc_n$, where $a_n\in \Z^{d_1}$, $c_n$ begins with a non-trivial element of $\Z^{d_2}$ and $a_1b_1...a_kb_ka_n$ converges to $\tilde{g}$.
The length of $c_n$ tends to infinity.
We will now show that $K(g,g_n)$ still converges to $K(g,\tilde{g})$, which is the limit defined in the first case.
Define the set $V_n=a_1b_1...a_kb_ka_nB(e,r(\mu))$.
For large enough $n$, it is a transitional set between $g$ and $g_n$ and also between $e$ and $g_n$.
Then,
$$K(g,g_n)=\frac{\sum_{v\in V_n}\mathds{P}(g\rightarrow v)\mathds{P}(v\rightarrow g_n;V_n^c)}{\sum_{v\in V_n}\mathds{P}(e\rightarrow v)\mathds{P}(v\rightarrow g_n;V_n^c)},$$
where $\mathds{P}(v,g_n;V_n^c)$ is the probability of going from $v$ to $g_n$ without passing through $V_n$, except at $v$.
Points in $V_n$ also converge to $\tilde{g}$ and stay in a bounded neighborhood of $a_1b_1...a_kb_k\Z^{d_1}$,
so that we can apply the previous case to those points.

Let $\epsilon>0$. For large enough $n$, for every $v$ in $V_n$,
$$\left |\frac{\mathds{P}(g\rightarrow v)}{\mathds{P}(e\rightarrow v)}-K(g,\tilde{g})\right |\leq \epsilon,$$
that is,
$$|\mathds{P}(g\rightarrow v)-\mathds{P}(e\rightarrow v)K(g,\tilde{g})|\leq \epsilon \mathds{P}(e\rightarrow v).$$
Indeed, since $V_n$ are balls of the same radius, all points in $V_n$ are a uniformly bounded distance away from each other.
Thus, the choice of large enough $n$ is uniform in $v\in V_n$.
Then,
\begin{align*}
    |K(g,g_n)-K(g,\tilde{g})|&=\left | \frac{\sum_{v\in V_n}(\mathds{P}(g\rightarrow v)-K(g,\tilde{g})\mathds{P}(e\rightarrow v))\mathds{P}(v\rightarrow g_n;V_n^c)}{\sum_{v\in V_n}\mathds{P}(e\rightarrow v)\mathds{P}(v\rightarrow g_n;V_n^c)}\right |\\
    &\leq \frac{\sum_{v\in V_n}\left |\mathds{P}(g\rightarrow v)-K(g,\tilde{g})\mathds{P}(e\rightarrow v)\right |\mathds{P}(v\rightarrow g_n;V_n^c)}{\sum_{v\in V_n}\mathds{P}(e\rightarrow v)\mathds{P}(v\rightarrow g_n;V_n^c)}\leq \epsilon.
\end{align*}
Thus, $K(g,g_n)$ converges to $K(g,\tilde{g})$, which concludes the proof.
\end{proof}

Similarly, one can prove the following.
\begin{prop}\label{freeproductMartinboundary3}
Let $(g_n)$ be a sequence in $\Z^{d_1}\star \Z^{d_2}$ and $g\in \Z^{d_1}\star \Z^{d_2}$.
Assume that $(g_n)$ converges in a $\Z^{d_2}$ factor to a limit
$\tilde{g}=a_1b_1...a_k\theta$.
Then, $K(g,g_n)$ converges to some quantity that we denote by $K(g,\tilde{g})$.
\end{prop}

\subsection{Continuity and separation}\label{Sectioncontinuityandseparationfreeproducts}
We defined $K(g,\tilde{g})$ for $g\in \Z^{d_1}\star \Z^{d_2}$ and $\tilde{g}$ in the $\mathrm{CAT}(0)$ compactification of $\Z^{d_1}\star \Z^{d_2}$.
What is left to do in order to prove Theorem~\ref{theorem1} (except for the minimality assertion, which will be proved in next section) is to show that functions $K(g,\tilde{g})$ are continuous with respect to $\tilde{g}$ and that if $\tilde{g}_1\neq \tilde{g}_2$,
one can find $g$ such that $K(g,\tilde{g}_1)\neq K(g,\tilde{g}_2)$.

Since $\Z^{d_1}\star \Z^{d_2}$ is dense into its compactification, to show continuity, one only has to deal with sequences $(g_n)$ in $\Z^{d_1}\star \Z^{d_2}$ that converge to some $\tilde{g}$ in the boundary.
This is contained in Propositions~\ref{freeproductMartinboundary1}, \ref{freeproductMartinboundary2} and \ref{freeproductMartinboundary3}.

To conclude, let us show that $\Z^{d_1}\star \Z^{d_2}$ separates points on the boundary.
Let $\tilde{g}_1\neq \tilde{g}_2$ be two points in the boundary and assume first that those two points correpond to two different ends, that is
\begin{enumerate}
    \item either $\tilde{g}_1$ is an infinite word and $\tilde{g}_2$ is in some $\Z^{d_i}$ factor,
    \item or $\tilde{g}_2$ is an infinite word and $\tilde{g}_1$ is in some $\Z^{d_i}$ factor,
    \item or $\tilde{g}_1$ and $\tilde{g}_2$ are in two different $\Z^{d_i}$ factors,
    \item or $\tilde{g}_1$ and $\tilde{g}_2$ are two different infinite words.
\end{enumerate}

\begin{lem}\label{lemmaseparationfreeproduct}
In those cases, if $(g_n)$ tends to infinity in the end of $\tilde{g}_1$, then $K(g_n,\tilde{g}_2)$ converges to 0, when $n$ tends to infinity.
\end{lem}

\begin{proof}
Let $(h_l)$ be a sequence converging to $\tilde{g}_2$.
Then, $K(g_n,\tilde{g}_2)$ is the limit of $K(g_n,h_l)$ when $l$ tends to infinity.
If $l$ and $n$ are large enough, one can find a transitional set $V$ between $g_n$ and $h_l$, and one can actually choose $V$ independent of $n$ and $l$.
Thus,
$$K(g_n,h_l)=\frac{p_{g_n}^V\cdot p_V^{h_l}}{\mathds{P}(e\rightarrow h_l)}.$$
Since the group $\Z^{d_1}\star \Z^{d_2}$ is non-amenable, if $v\in V$, $\mathds{P}(g_n\rightarrow v)$ converges to 0.
In particular, $p_{g_n}^V(v)$ converges to 0.
Let $\epsilon>0$. For large enough $n$,
$$K(g_n,h_l)\leq \frac{\sum_{v\in V}p_V^{h_l}(v)}{\mathds{P}(e\rightarrow h_l)}\epsilon.$$
Besides, one can find $\alpha$ such that for every $v\in V$, $p_V^{h_l}(v)\leq \alpha \mathds{P}(e\rightarrow h_l)$.
Indeed,
$$p_V^{h_l}(v)\leq \mathds{P}(v\rightarrow h_l)\leq \frac{1}{\mathds{P}(e\rightarrow v)}\mathds{P}(e\rightarrow h_l).$$
Thus, for large enough $n$, $K(g_n,h_l)\leq \alpha \epsilon$,
so that $K(g_n,\tilde{g}_2)\leq \alpha \epsilon$.
\end{proof}

Using Proposition~\ref{theoremseparation}, we can show that if $\tilde{g}_1$ lies in some $\Z^{d_i}$ factor, there exists a sequence $(g_n)$ lying in the $\Z^{d_i}$ factor such that
$K(g_n,\tilde{g}_1)$ tends to infinity.
Since $K(g_n,\tilde{g}_2)$ converges to 0, one can find $n$ such that those two quantities differ.
Thus, there exists $g$ such that $K(g,\tilde{g}_1)\neq K(g,\tilde{g}_2)$.
Similarly, one can find such a $g$ if $\tilde{g}_2$ lies in some $\Z^{d_2}$ factor.
This shows that $\Z^{d_1}\star \Z^{d_2}$ separates points on the boundary in the three first cases.

Assume now that $\tilde{g}_1$ and $\tilde{g}_2$ are two different infinite words.
Let $(h_l)$ be a sequence converging to $\tilde{g}_1$.
Fix transitional sets $V_1,...,V_{\varphi(l)}$ between $e$ and $h_l$, with $\varphi(l)$ that tends to infinity.
Then, for every integer $n$, $\mathds{P}(e\rightarrow h_l)=p_{e}^{V_n}\cdot p_{V_n}^{h_l}$.
Define then a sequence $(g_n)$ choosing one $g_n$ in each $V_n$.
Proposition~\ref{propalpha1} shows that there exists $\alpha>0$ such that for every $v\in V_n$, $p_{e}^{V_n}(v)\leq \frac{1}{\alpha} \mathds{P}(e\rightarrow g_n)$.
In particular, if $l$ is large enough,
$$K(g_n,h_l)\geq \frac{\alpha}{\mathds{P}(e\rightarrow g_n)},$$
and so
$$K(g_n,\tilde{g}_1)\geq \frac{\alpha}{\mathds{P}(e\rightarrow g_n)}.$$
Since $(g_n)$ converges to $\tilde{g}_1$ and since the random walk is transient, $\mathds{P}(e\rightarrow g_n)$ converges to 0.
Thus, $K(g_n,\tilde{g}_1)$ tends to infinity and $K(g_n,\tilde{g}_2)$ converges to 0, which suffices to conclude as above.

To end the proof of Theorem~\ref{theorem1}, we have to deal with $\tilde{g}_1\neq \tilde{g}_2$ in the boundary, lying on the same $\Z^{d_i}$ factor.
But in that case, we only have to deal with a transition kernel on $\Z^d\times \{1,...,N\}$ and we can apply
Proposition~\ref{theoremseparation}.
\qed

\medskip
We thus proved Theorem~\ref{theorem1}.
The proof of Theorem~1.1' is similar.
Since a finitely generated virtually abelian group can be identified with $\Z^d\times L$, where $L$ is finite (see Section~\ref{SectionGeometricboundaries}), we can apply the same strategy, replacing $\Z^{d_i}\times \{1,...,N\}$ with
$\Gamma_i\times \{1,...,N\}$ and then identifying $\Gamma_i\times \{1,...,N\}$ with $\Z^{d_i}\times \{1,...,N'\}$ for some other integer $N'$.
We can use Proposition~\ref{generalabeliangroups} to conclude. \qed

\medskip
As stated above, free products $\Gamma_1\star \Gamma_2$ are hyperbolic groups relative to the subgroups $\Gamma_1$ and $\Gamma_2$.
A first generalization of our results would be to more general relatively hyperbolic groups.
Such groups arise in nature in many situations.
For example, the fundamental group of a geometrically finite manifold of negative pinched sectional curvature is hyperbolic relative to its cusp subgroups.
In constant curvature, the cusp subgroups are virtually abelian, but in variable cuvature, they can be virtually nilpotent.
That is the reason for mentioning nilpotent groups in the introduction.
The first step would be to identify the Martin boundary of such groups (for non-centered transition kernels).


\section{Minimal Martin boundary}\label{SectionMinimalMartinboundary}
In this last section, we prove that the Martin boundary is minimal in every situation encountered above, ending the proofs of Theorems~\ref{theorem1} and~\ref{theorem2}.
For any detail on the minimal Martin boundary, we refer to \cite{Sawyer}.

\begin{prop}\label{minimalMartin1}
The Martin boundary of an irreducible, $\Z^d$-invariant, finitely supported Markov chain on $\Z^d\times \{1,...,N\}$ is minimal.
\end{prop}

\begin{prop}\label{minimalMartin2}
The Martin boundary of an irreducible, finitely supported random walk on the free product $\Z^{d_1}\star \Z^{d_2}$ is minimal.
\end{prop}

To prove these propositions, we will need the two results below.
They are refinements of separation properties we used (Proposition~\ref{theoremseparation} and lemma~\ref{lemmaseparationfreeproduct}).

\begin{prop}\label{prop1minimalMartin}
Let $p$ be a strongly irreducible transition kernel on $\Z^d\times \{1,...,N\}$ which is $\Z^d$-invariant and satisfies Assumptions~\ref{hyp1} and \ref{hyp2} of section~\ref{SectionSub-Markovchainsonathickenedlattice}.
Let $\tilde{y}_1\neq \tilde{y}_2$ be two points in the Martin boundary $\partial \Z^d$.
There exists a neighborhood $U$ of $y_1$ in $\partial \Z^d$ and a sequence $((x_n,k_n))$ of $\Z^d\times \{1,...,N\}$ such that
for every $\tilde{y}$ in $U$, $K((x_n,k_n),\tilde{y})$ tends to infinity, uniformly in $\tilde{y}$ and $K((x_n,k_n),\tilde{y}_2)$ converges to 0.
\end{prop}

The proof is exactly the same as the proof of Proposition~\ref{theoremseparation}, as uniformity is obvious.

\begin{prop}\label{prop2minimalMartin}
Consider an irreducible, finitely supported random walk on the free product $\Z^{d_1}\star \Z^{d_2}$.
Let $\tilde{g}_1$ and $\tilde{g}_2$ be two points in the Martin boundary that correspond to two different ends.
Then, there exists a neighborhood $U$ of $\tilde{g}_1$ in the Martin compactification $\Z^{d_1}\star \Z^{d_2}\cup \partial (\Z^{d_1}\star \Z^{d_2})$ and a constant $M\geq 0$ such that
for every $g\in U\cap (\Z^{d_1}\star \Z^{d_2})$, $K(g,\tilde{g}_2)\leq MG(g,e)$.
\end{prop}

\begin{proof}
Let $(g_n)$ be a sequence converging to $\tilde{g}_2$.
Let $U$ be a neighborhood of $\tilde{g}_1$ such that for every $g\in U\cap (\Z^{d_1}\star \Z^{d_2})$, there exists a fixed transitional set $V$ between $g$ and $g_n$, for large enough $n$
(see Definition~\ref{deftransitionalset} for transitional sets).
Fix $v_0\in V$.
According to Lemma~\ref{lemmatransitionalsets}, the random walk from $g\in U\cap (\Z^{d_1}\star \Z^{d_2})$ to $g_n$ has to go through $V$, so we have
$$\mathds{P}(g\rightarrow g_n)=\sum_{v\in V}\mathds{P}(g\rightarrow v;V^c)\mathds{P}(v\rightarrow g_n),$$
where $\mathds{P}(g\rightarrow v;V^c)$ is the probability to go from $g$ to $v$ without passing through $V$ before $v$.
Since the size of $V$ is fixed, we also have
$$\mathds{P}(g\rightarrow v;V^c)\leq M_0\mathds{P}(g\rightarrow v_0)$$ and
$$\mathds{P}(v\rightarrow g_n)\leq M_1\mathds{P}(v_0\rightarrow g_n),$$
for some $M_0,M_1\geq0$.
Thus,
$$\mathds{P}(g\rightarrow g_n)\leq M_2\mathds{P}(g\rightarrow v_0)\mathds{P}(v_0\rightarrow g_n).$$
Now, $\mathds{P}(g\rightarrow v_0)\mathds{P}(v_0\rightarrow e)\leq \mathds{P}(g\rightarrow e)$ and $\mathds{P}(e\rightarrow v_0)\mathds{P}(v_0\rightarrow g_n)\leq \mathds{P}(e\rightarrow g_n)$, so that
$\mathds{P}(g\rightarrow g_n)\leq M_3\mathds{P}(g\rightarrow e)\mathds{P}(e\rightarrow g_n)$ and thus $K(g,g_n)\leq M_3\mathds{P}(g\rightarrow e)$.
All these inequalities are satisfied for large enough $n$, so $K(g,\tilde{g}_2)\leq M_3\mathds{P}(g\rightarrow e)$.
\end{proof}

We will also need the two following lemmas.

\begin{lem}\label{lemma1minimalMartin}
Let $p$ is a finitely supported transition kernel on a countable space $E$ which is transient and irreducible.
Then for every $\tilde{y}$ in the Martin boundary, the Martin kernel $K(\cdot,\tilde{y})$ is a harmonic function.
\end{lem}

\begin{proof}
Recall that everything is defined up to the choice of a base point $x_0\in E$.
If $y\in E$, the Green function $G(\cdot,y)$ is harmonic everywhere except at $y$ and so is the function $\frac{G(\cdot, y)}{G(x_0,y)}$.
Let $x\in E$ and let $y_n$ converge to $\tilde{y}$.
Then, for large enough $n$, the function $\phi_n=\frac{G(\cdot, y_n)}{G(x_0,y_n)}$ is harmonic at $x$.
Now, $\phi_n$ converges pointwise to $K(\cdot, \tilde{y})$, so
$\sum_{z\in E}\phi_n(z)p(x,z)$ converges to $\sum_{z\in E}K(z,\tilde{y})p(x,z)$, since $p$ has finite support.
The first sum is equal to $\phi_n(x)$ for large enough $n$ and thus also converges to $K(x,\tilde{y})$, so that $K(\cdot,\tilde{y})$ is harmonic at $x$ and this holds for every $x$.
\end{proof}

\begin{lem}\label{lemma2minimalMartin}
Let $p$ be a finitely supported transition kernel on a countable space $E$ which is transient and irreducible.
Let $\phi$ be a non-negative harmonic function on $E$ and $\mu_{\phi}$ the corresponding measure on the minimal Martin boundary $\partial_mE$.
For $\mu_{\phi}$-almost every point $\tilde{y}$ in $\partial_mE$, $\frac{G(y_n,x_0)}{\phi(y_n)}$ converges to 0 when $y_n$ converges to $\tilde{y}$.
\end{lem}

\begin{proof}
This is \cite[Proposition~II.1.6]{Ancona2}.
\end{proof}

We can now prove Propositions~\ref{minimalMartin1} and~\ref{minimalMartin2}.
The demonstrations will follow the same strategy as the one presented in \cite{Haissinsky} (see Proposition~3.7 there).

\begin{proof}[Proof of Proposition~\ref{minimalMartin1}]
We consider a transition kernel on $\Z^d\times \{1,...,N\}$ having the properties stated in Proposition~\ref{minimalMartin1}.
First, in the centered case, the Martin boundary and the minimal Martin boundary are trivial, according to Proposition~\ref{trivialMartincentered}, so we can assume that the chain in non-centered.
Also, using Lemma~\ref{Spitzertrick}, we can assume that the chain is strongly irreducible.
Let $\tilde{y}$ be in the Martin boundary.
Then the Martin kernel $K(\cdot,\tilde{y})$ is harmonic.
Let $\mu$ be the corresponding measure on $\partial_m\Z^d$.
We will prove that $\{\tilde{y}\}$ is the support of $\mu$.
Let $\tilde{z}\neq \tilde{y}$ be in the Martin boundary and let $U$ be a neighborhood of $\tilde{z}$ and $((x_n,k_n))$ a sequence given by Proposition~\ref{prop1minimalMartin}, so that for every $\tilde{z}'$ in $U$, $K((x_n,k_n),\tilde{z}')$ tends to infinity, uniformly in $\tilde{z}'$, while $K((x_n,k_n),\tilde{y})$ converges to 0.
By definition,
$$K((x_n,k_n),\tilde{y})=\int_{\partial_m\Z^d}K((x_n,k_n),\tilde{w})\mathrm{d}\mu(\tilde{w})\geq \int_{U}K((x_n,k_n),\tilde{w})\mathrm{d}\mu(\tilde{w}),$$
so that $K((x_n,k_n),\tilde{y})\geq \mu(U)$ for large enough $n$ and so $\mu(U)=0$.
This proves that $\tilde{z}$ is not in the support of $\mu$.
Since this holds for every $\tilde{z}\neq \tilde{y}$, the support is reduced to $\{\tilde{y}\}$.
Notice that we did not actually prove that $\tilde{y}$ is in the support of $\mu$, but this follows from the fact that every other point cannot be in it, and $\mu$ is not zero, since $K(\cdot, \tilde{y})$ is not zero.
In particular, $\tilde{y}\in \partial_m\Z^d$ and this holds for every $\tilde{y}$.
\end{proof}

\begin{proof}[Proof of Proposition~\ref{minimalMartin2}]
We consider a random walk on $\Z^{d_1}\star \Z^{d_2}$ with the properties stated in Proposition~\ref{minimalMartin2}
and reduce to the case of a strongly irreducible chain, using Lemma~\ref{Spitzertrick}.
Let $\tilde{g}$ be in the Martin boundary and let $\mu$ be the corresponding measure on the minimal Martin boundary $\partial_m(\Z^{d_1}\star \Z^{d_2})$.
Again, we will prove that every other point $\tilde{h}$ in the Martin boundary cannot be in the support of $\mu$, which will be enough to conclude.
First, let $\tilde{h}$ be a point in the Martin boundary such that $\tilde{g}$ and $\tilde{h}$ correspond to different ends.
Proposition~\ref{prop2minimalMartin} shows that there is a neighborhood $U$ of $\tilde{h}$ such that for $g$ in $U\cap \Z^{d_1}\star \Z^{d_2}$, $K(g,\tilde{g})\leq MG(g,e)$.
However, Lemma~\ref{lemma2minimalMartin} shows that for $\mu$-almost every point $\tilde{h}'$ in $\partial_m(\Z^{d_1}\star \Z^{d_2})$, $\frac{G(g_n,e)}{K(g_n,\tilde{g})}$ converges to 0 when $g_n$ converges to $\tilde{h}'$.
This proves that $\tilde{h}$ is not in the support of $\mu$.
Now, if $\tilde{h}$ and $\tilde{g}$ correspond to the same end, then $\tilde{g}$ necessarily lies in a $\Z^{d_i}$ factor, and so does $\tilde{h}$, so that we can use Proposition~\ref{prop1minimalMartin} to conclude.
\end{proof}

\bibliographystyle{plain}
\bibliography{free_product}

\end{document}